\documentclass[11pt]{article}
\usepackage[a4paper,left=25mm,top=20mm,right=20mm,bottom=15mm]{geometry}

\usepackage{graphicx}
\usepackage{algorithm}
\usepackage[noend]{algpseudocode}
\usepackage{amsmath, amsfonts, amssymb, amsthm}
\usepackage{mathtools}
\usepackage[colorlinks]{hyperref}
\hypersetup{
	citecolor=blue,
   linkcolor=blue,
}

\usepackage{authblk}
\usepackage[numbers,sort&compress]{natbib}
\usepackage[font=footnotesize,labelfont=bf]{caption}

\algrenewcommand\algorithmicrequire{\textbf{Input:}}
\algrenewcommand\algorithmicensure{\textbf{Output:}}

\newtheorem{theorem}{Theorem}[section]
\newtheorem{lemma}[theorem]{Lemma}%[section]
\newtheorem{definition}[theorem]{Definition}%[section]
\newtheorem{cor}[theorem]{Corollary}%[section]

\newcommand{\tmp}{\ensuremath{\operatorname{tmp}}}
\newcommand{\cl}{\ensuremath{\operatorname{cl}}}
\newcommand{\spa}{\ensuremath{\operatorname{span}}}
\newcommand{\Amax}{\ensuremath{A^{\star}}}
\newcommand{\Bmax}{\ensuremath{B^{\star}}}
\newcommand{\Rn}{\ensuremath{R_{\operatorname{new}}}}
\newcommand{\LT}{\ensuremath{\mathcal{L}}}
\newcommand{\SC}{\ensuremath{\mathfrak{sc}}}
\newcommand{\R}{\ensuremath{\mathcal{R}}}
\newcommand{\minimal}{\ensuremath{\operatorname{min}}}
\newcommand{\minimum}{\ensuremath{\operatorname{MIN}}}
\newcommand{\rt}[1]{\ensuremath{\mathtt{#1}}}
\newcommand{\mc}{\mathcal}

\newcommand{\LrR}[2]{\ensuremath{\mathfrak{L}_{#1}(#2)}}
\newcommand{\lmax}[2]{\ensuremath{\ell^{\star}_{#1}(#2)}}
\newcommand{\Lmax}[2]{\ensuremath{\mathfrak{L}^{\star}_{#1}(#2)}}
\newcommand{\RF}{\ensuremath{(R,\mathbb{F}_R)}}

\providecommand{\keywords}[1]{\textbf{\textit{Keywords: }} #1}

\makeindex             % used for the subject index
\begin{document}

\title{The Matroid Structure of Representative Triple Sets and Triple-Closure Computation}

\author[]{Marc Hellmuth}

\author[]{Carsten R.\ Seemann}

\affil[]{\footnotesize Dpt.\ of Mathematics and Computer Science, University of Greifswald, Walther-
  Rathenau-Strasse 47, D-17487 Greifswald, Germany\\ \smallskip

	Saarland University, Center for Bioinformatics, Building E 2.1, P.O.\ Box 151150, D-66041 Saarbr{\"u}cken, Germany\\ \smallskip
	Email: \texttt{mhellmuth@mailbox.org}}

\date{\ }

\maketitle

\abstract{ 
The closure $\cl(R)$ of a consistent set $R$ of triples (rooted binary trees on three leaves) provides essential information about tree-like relations that are shown by any supertree that displays all triples in $R$. In this contribution, we are concerned with representative triple sets, that is, subsets $R'$ of $R$ with $\cl(R') = \cl(R)$. In this case, $R'$ still contains all information on the tree structure implied by $R$, although $R'$ might be significantly smaller. We show that representative triple sets that are minimal w.r.t.\ inclusion form the basis of a matroid. This in turn implies that minimal representative triple sets also have minimum cardinality. In particular, the matroid structure can be used to show that minimum representative triple sets can be computed in polynomial time with a simple greedy approach. For a given triple set $R$ that ``identifies'' a tree, we provide an exact value for the cardinality of its minimum representative triple sets. In addition, we utilize the latter results to provide a novel and efficient method to compute the closure $\cl(R)$ of a consistent triple set $R$ that improves the time complexity $\mc{O}(|R||L_R|^4)$ of the currently fastest known method proposed by Bryant and Steel (1995). In particular, if a minimum representative triple set for $R$ is given, it can be shown that the time complexity to compute $\cl(R)$ can be improved by a factor up to $|R||L_R|$. As it turns out, collections of quartets (unrooted binary trees on four leaves) do not provide a matroid structure, in general. 
}

\smallskip
\noindent
\sloppy

\section{Introduction}

Inference of phylogenetic relationships between genes or species based on
genomic sequence information is one of the main issues in phylogenomics
\cite{Philippe2007}. The evolutionary history of genes and species is
usually represented as a tree. One of the possible building blocks for the reconstruction
of the histories of both, genes and species, are provided by triples
(rooted binary trees on three leaves)
\cite{HHH+12,HHH+13,LDEM:16,lafond2015orthology,DEML:16,GS:07,JW:06,IKK+09,Hellmuth2017,SEMPLE03,BYRKA20101136}. Such
triples can be obtained directly from sequence data and are combined to a
``supertree'' that provides then the information of the history of the
respective genes or species \cite{Hellmuth:15a,HW:16b,GASWH:18,Lafond2014,CK:99,CK:01,KSH:16,Ewing2008,DD:10,KANNAN199626,csuros2002fast}. 
In this contribution, we consider \emph{consistent} sets $R$ of triples, that is,
all triples of $R$ fit into a common supertree, which enforces further
tree-like relations to hold \cite{GSS:07,Dekker86,BS:95}. This allows one
to define a \emph{closure operation} $\cl(R)$ for $R$ that comprises all
triples that are displayed by every proper supertree for $R$. The closure of
sets of rooted or unrooted trees has been extensively studied in the last
decades \cite{Dekker86,BS:95, GSS:07,
Bryant97,huber2005recovering,BBDS:00,Dezulian2004} and has various
applications in phylogenomics
\cite{RBC+07,Hellmuth:15a,SBR:11,RSA:RSA3,HDKS:04,ERDOS199977,Meacham1983,Mossel2004189,WCT:04}. 

Here, we are particularly interested in the computation of the closure and
\emph{representative} sets $R'$ for consistent triple sets $R$, that is,
subsets $R'$ of $R$ that satisfy $\cl(R') = \cl(R)$. Such representative
sets $R'$ are of particular interest, since on the one hand, they can
reduce the space complexity to store all information on the tree-like
relationships that is also provided by $R$ and, on the other hand, will
significantly improve the time complexity to compute the closure, as we
shall see later. Natural optimization problems within this context aim at
finding representative sets $R'$ that are minimal w.r.t.\ inclusion or have
minimum size among all representative subsets of $R$. Gr\"unewald, Steel
and Swenson \cite{GSS:07} established important results to the latter
problems. In particular, they characterized minimal representative triple
sets $R'\subseteq R$ for the case that $R$ ``identifies'' a given tree $T$
and gave lower bounds $B(T)$ on their cardinalities. 
Moreover, Mike Steel showed that all minimal ``tree-defining'' sets
of rooted triples must have the same size \cite{Steel1992}.
However, for an arbitrary consistent triple set $R$
it is still  unclear whether  the (decision version of the) problem of finding a representative
subset $R'\subseteq R$ of minimum size is NP-complete or polynomial-time solvable.

In this contribution, 
we show that minimum representative subset $R'\subseteq R$ can
be computed in polynomial time. To this end, we show that minimal
representative sets $R'\subseteq R$ form the basis of the matroid
$(R,\mathbb{F}_R)$ \cite{Oxley:11,korte2012combinatorial}. Since all basis
elements of a matroid have the same size and since minimum representative
sets are minimal, it turns out that minimum representative sets can be
computed with a simple greedy algorithm. We emphasize that there is a clear
difference between the closure operator $\cl(R)$ for rooted triple sets $R$
and the respective matroid closure operator, although $\cl(R)$ is used to
define the matroid $(R,\mathbb{F}_R)$, see \cite{Bryant97} or Section
\ref{sec:min} for further details. We exploit the techniques we used to
prove the matroid structure and provide a novel algorithm to compute the
closure $\cl(R)$ of a consistent set $R$ of triples. Let $L_R$ denote the
set of leaves on which $R$ is defined on. If $R$ is large sized, that is,
$|R|=\Theta(L_R^3)$, then our algorithm has the same asymptotic time
complexity as the method proposed by Bryant and Steel which runs in
$\mc{O}(|R||L_R|^4) \subseteq \mc{O}(|R|^5)$ time \cite{BS:95}. However,
our algorithm has a time complexity of $\mc{O}(|R|^2|L_R|) \subseteq
\mc{O}(|R|^3)$ and thus, significantly improves the computational effort
for moderately sized input triple sets $R$. Further runtime improvements
(up to a factor of $|R||L_R|$) can be achieved whenever minimum
representative subset $R'\subseteq R$ are used as input triple set. It
should be noted that Bryant and Steel established this algorithm in order
to show that $\cl(R)$ can be computed in polynomial-time rather than to be
efficient. Nevertheless, they supposed that ``a far more efficient
algorithm could be found''. However, over the last two decades no such
algorithm appeared in the literature. 
We wish to point out that 
the theory of matroids  has touched phylogenetics also in many other contexts, see e.g.\ 
\cite{AK:06,dmtcs:2078,MOULTON20091496,Steel:05,MOULTON2007186,PG:05,Ardila2005,Gusfield:2002,Francisco2009}.

This contribution is organized as follows: In Section \ref{sec:prelim}, we
present the basic and relevant concepts used in this paper. In particular,
we review important results for closure operations on rooted triple sets
established by Bryant and Steel \cite{BS:95,Bryant97}. A key property that
will play a major role in this paper is provided by the graph
representation of triple sets (Ahograph)
and its connected components. In
Section \ref{sec:repT}, we are concerned with structural properties of
representative subsets $R'\subseteq R$ that are closely related to the
structure of the Ahograph. The latter results will be used in Section
\ref{sec:min} to show that minimal representative sets $R'\subseteq R$ (and
its subsets) form a matroid $(R,\mathbb{F}_R)$. In Section
\ref{sec:cl}, we present a novel method to compute the closure $\cl(R)$. 
Finally, we discuss in Section \ref{sec:fr} further results. 
We give sufficient conditions that are quite useful to check whether
an arbitrary triple is contained in all minimal representative sets and if $R$
is already minimal. 
Moreover, we review and generalize some of the results established for 
triple sets $R$ that ``identify'' or ``define'' a tree. In addition, 
we address the problem of finding minimal representative sets $Q'\subseteq Q$
of a collection $Q$ of quartets (unrooted binary tree on four leaves). 
As it turns out, such sets do not provide a matroid structure.
We conclude with a short discussion about the established results and open problems in
Section \ref{sec:end}.

\section{Preliminaries}
\label{sec:prelim}

We consider undirected graphs $G=(V,E)$ with non-empty vertex set $V$ and
edge set $E$. A graph $G=(V,E)$ is \emph{connected} if for any two vertices
$x,y\in V$ there is a sequence of vertices $(x,v_1,\dots,v_n,y)$, called
\emph{walk}, such that the edges $(x,v_1),(v_n,y)$ and $(v_i,v_{i+1})$,
$1\leq i\leq n-1$ are contained in $E$. A walk $(x,v_1,\dots,v_n,y)$ in
which all vertices are pairwise distinct is called a \emph{path} and
denoted by $P_{xy}$. A \emph{cycle} is a walk $(x,v_1,\dots,v_n,x)$
for which $n\geq 2$ and $(x,v_1,\dots,v_n)$ is a path.
A graph $H=(W,F)$ is a \emph{subgraph} of $G=(V,E)$,
in symbols $H\subseteq G$, if $W\subseteq V$ and $F\subseteq E$. The
subgraph $H=(W,F)$ is an \emph{induced} subgraph of $G=(V,E)$, if $x,y\in
W$ and $(x,y)\in E$ implies $(x,y)\in F$. If $H=(W,F)$ is an induced
subgraph of $G$ we write $\langle W \rangle_G$ or simply $\langle W
\rangle$ if there is no risk of confusion. A \emph{connected component} of
a graph $G=(V,E)$ is a subset $W\subseteq V$ such that $\langle W
\rangle_G$ is connected and maximal w.r.t.\ inclusion.

A \emph{tree} $T = (V, E)$  is a connected graph that does not contain cycles. 
The \emph{leaf set} $L \subseteq V$ of $T$ comprises all vertices that have degree 1.
The vertices that are contained in  $V^0\coloneqq V \setminus L$ are called \emph{inner} vertices.
The set of inner edges $E^0$ contains all edges $(x,y)\in E$ for which $x,y\in V^0$.
A \emph{rooted tree} $T = (V, E)$  is a tree 
with one distinguished inner vertex $\rho_T \in V$ called \emph{root} of $T$.
If every inner vertex of an unrooted tree has degree 3, the tree is called \emph{binary}.
A rooted tree is called  \emph{binary} if the degree of each inner vertex $v\neq \rho_T$
is 3 and the degree of the root $\rho_T$ is 2.
In what follows, we consider rooted trees $T=(V,E)$ 
such that all inner vertices that are distinct from the root have degree at least three. 
For every vertex $v\in V$ we denote by $C(v)$ 
the leaf set of the subtree of $T$ rooted at $v$
and put $\mathcal{C}(T)=\bigcup_{v\in V}\{C(v)\}$,  called the \emph{hierarchy of $T$}.
We say that a rooted tree $T'$ \emph{refines} 
$T$, in symbols $T \le T'$, if $\mathcal{C}(T)\subseteq \mathcal{C}(T')$. 

A \emph{triple $\rt{ab|c}$}  is a binary rooted tree $T$ on three leaves $a,b$ and $c$ such that
     the path from $a$ to $b$ does not intersect the path
from $c$ to the root $\rho_T$.
A rooted tree $T$ with leaf set $L$ \emph{displays} a triple $\rt{ab|c}$,
if $a,b,c\in L$ and the path from $a$ to $b$ does not intersect the path
from $c$ to the root $\rho_T$. Note, that no distinction is made between
$\rt{ab|c}$ and $\rt{ba|c}$. The set of all triples
that are displayed by the rooted tree $T$ is denoted by $\mc{R}(T)$. An
arbitrary collection $R$ of triples is called \emph{triple set}. 
A triple set $R$ is  \emph{consistent} if there is a rooted tree $T$ such that
$R\subseteq \mc{R}(T)$. In the latter case, we say that $T$ \emph{displays}
$R$. 
The set $L_R:=\cup_{\rt{ab|c}\in R} \{a,b,c\}$
is the union of the leaf set of each triple in $R$. 
A triple set $R$ \emph{identifies} a rooted tree $T$ with leaf set
$L_R$, if $T$ displays $R$ and any other tree that displays $R$ refines
$T$. A triple set $R$ \emph{defines} a rooted tree $T$ with leaf set $L_R$,
if $T$ is the unique tree (up to isomorphism) that displays $R$.

There is a polynomial-time algorithm, which is customarily referred to as
\texttt{BUILD} \cite{Semple:book,Steel:book}, that was established by 
Aho, Sagiv, Szymanski, and Ullman \cite{Aho:81}. 
\texttt{BUILD} either constructs a
rooted tree $T$ that displays $R$ or recognizes that $R$ is inconsistent
\cite{Aho:81}. The runtime of \texttt{BUILD} is $\mc{O}(|L_R||R|)$
\cite{Semple:book}. Further practical implementations and improvements have
been discussed in \cite{Jansson:05,DF:16,Henzinger:99,Holm:01}.
\texttt{BUILD} is a top-down, recursive algorithm \cite{Aho:81,BS:95} 
that uses an auxiliary graph that is also known as \emph{Ahograph} \cite{huson2010phylogenetic}, 
\emph{clustering graph} \cite{Semple:book} or \emph{cluster graph} \cite{DKK+12}.
We will use the term ``Ahograph''.  This graph is used to represent the
structure of a collection of triples: For a given triple set $R$ and an
arbitrary subset $\LT\subseteq L_R$, the Ahograph $[R,\LT]$ has vertex set
$\LT$ and two vertices $a,b\in \LT$ are linked by an edge, if there is a
triple $\rt{ab|c} \in R$ with $c\in \LT$. Based on connectedness properties
of the graph $[R,\LT]$ for particular subsets $\LT\subseteq L_R$, the
algorithm \texttt{BUILD} determines whether $R$ is consistent or not. In
particular, this algorithm makes use of the following well-known theorem.

\begin{theorem}[\cite{Aho:81,BS:95}]
A set  of triples $R$ is consistent if and only if for each
subset $\LT\subseteq L_R$ with $|\LT|>1$ the graph $[R,\LT]$ is disconnected. 
\label{thm:ahograph}
\end{theorem}

Since we will use the Ahograph and its key features as a frequent tool in upcoming proofs, 
we now summarize some of its basic properties. 

\begin{lemma}[\cite{BS:95}]
  If $R'$ is a subset of the triple set $R$ and  $L'\subseteq L\subseteq L_R$,
  then $[R',L']$ is a subgraph of $[R,L]$. 
  \label{lem:subgraph}
\end{lemma}

\begin{lemma}
	Let $R$ be a triple set and $\LT\subseteq L_R$. 
	Assume that $A\subseteq \LT_A\subseteq \LT$ and $B\subseteq \LT_B\subseteq \LT$ 
   such that the induced subgraphs $\langle A \rangle_{[R,\LT_A]}$  and $\langle B \rangle_{[R,\LT_B]}$
	in $[R,\LT_A]$ and $[R,\LT_B]$, respectively,  are connected. If 
	$A\cap B\neq \emptyset$,  
	then $\langle A\cup B \rangle_G$ is connected in $G$, where $G=[R,\LT_A\cup \LT_B]$ or
	 $G = [R,\LT]$.
	\label{lem:cc}
\end{lemma}
\begin{proof}
	Let $A\subseteq \LT_A\subseteq \LT$, $B\subseteq \LT_B\subseteq \LT$ and assume that the induced subgraphs
	 $\langle A \rangle$  of $[R,\LT_A]$  and $\langle B \rangle$ of $[R,\LT_B]$ are connected.
	By Lemma \ref{lem:subgraph}, $[R,\LT_A]$ and $[R,\LT_B]$ are subgraphs of $[R,\LT_A\cup \LT_B]$.
	Let $x\in A\cap B$. Thus, every vertex $y\in A\cup B$
	is reachable from $x$ by a walk in $[R,\LT_A\cup \LT_B]$.
   Hence, any two vertices
	$y,y'\in A\cup B$ are reachable by a walk (over $x$) in $[R,\LT_A\cup \LT_B]$
   and therefore, $\langle A\cup B \rangle_{[R,\LT_A\cup \LT_B]}$ is a connected subgraph in 
	$[R,\LT_A\cup \LT_B]$. Since $\LT_A,\LT_B\subseteq \LT$ we can apply Lemma 
	\ref{lem:subgraph} and conclude that  $[R,\LT_A\cup \LT_B]$ is a subgraph of $[R,\LT]$ from what
	 the statement follows. 
\end{proof}

The requirement that a set $R$ of triples is consistent, and thus, that
there is a tree displaying all triples, allows to infer new triples from
the trees that display $R$ and to define a
\emph{closure operation} for $R$. 
Let $\spa(R)$ be the set
of all rooted trees with leaf set $L_R$ that display  $R$.
The closure of a consistent triple set $R$ is defined as
\[\cl(R) = \bigcap_{T\in \spa(R)} \mathcal{R}(T).\]
Hence, a triple $r$ is contained in the closure $\cl(R)$ if all
trees that display $R$ also display $r$. 
This operation satisfies the usual three properties of a closure operator \cite{BS:95},
namely: 
\begin{itemize}	
	\item  $R \subseteq \cl(R)$,  %\vspace{-0.1in}
	\item $\cl(\cl(R))=\cl(R)$, and %\vspace{-0.1in}
	\item if $R' \subseteq R$, then $\cl(R')\subseteq \cl(R)$. 
\end{itemize}	
There is a simple polynomial time algorithm to compute the closure that is
based on the following lemmas.
\begin{lemma}[{\cite[Prop.\ 9(1)]{BS:95}}]
	Let $R$ be a consistent triple set. 
	If $\cl(R)$ does not contain any triples with leaves $\{a,b,c\}$, 
	then $\cl(R)\cup\{\rt{ab|c}\}$, $\cl(R)\cup\{\rt{ac|b}\}$ and $\cl(R)\cup\{\rt{bc|a}\}$ are all consistent.
	\label{lem:bs}
\end{lemma}

\begin{lemma}
	Let $R$ be consistent.  For all 	$\{a,b,c\} \in \binom{L_R}{3}$ exactly 
	one of $R\cup\{\rt{ab|c}\}$, $R\cup\{\rt{ac|b}\}$ and $R\cup\{\rt{bc|a}\}$ is consistent
	(say  $R\cup\{\rt{ab|c}\}$) if and only if $\rt{ab|c}\in \cl(R)$.
	\label{lem:algo-cl}
\end{lemma}
\begin{proof}
	Assume that only $R\cup\{\rt{ab|c}\}$ is consistent while  
	$R\cup\{\rt{ac|b}\}$ and $R\cup\{\rt{bc|a}\}$ are not. 
	Since the latter two sets are not consistent, 
	there is no tree that displays $R$ and, in addition, $\rt{ac|b}$, resp., $\rt{bc|a}$.
	Thus, $\rt{ac|b}, \rt{bc|a}\notin \cl(R)$.
	Assume for contradiction that additionally $\rt{ab|c}\notin \cl(R)$.
	Hence, $\cl(R)$ does not contain any triples with the leaves $\{{a,b,c}\}$. 
	Lemma \ref{lem:bs} implies that 
	$\cl(R)\cup  \{\rt{ac|b}\}$ is consistent. However, this implies that
	there is a tree $T$ that display all triples of $\cl(R)$ and the triple $\rt{ac|b}$. 
	Since $R\subseteq \cl(R)$ this tree $T$ displays $R\cup\{\rt{ac|b}\}$;
 a contradiction.

	Conversely, let $\rt{ab|c}\in \cl(R)$.
	Thus, every tree that displays $R$ must also display $\rt{ab|c}$.
	Therefore, any tree that displays $R$ does not display $\rt{ac|b}$ and $\rt{bc|a}$.
	Hence, there is no tree that displays $R$ and in addition, $\rt{ac|b}$ 
	(resp.\ $\rt{bc|a}$), which implies that  
	$R\cup\{\rt{ac|b}\}$ and $R\cup\{\rt{bc|a}\}$ are not consistent. 
\end{proof}

Based on the latter result, the closure of a given consistent set $R$ can
be computed in $\mc{O}(|R||L_R|^4)$ time \cite{BS:95} as follows: For any
three distinct leaves $a,b,c\in L_R$ test whether exactly one of the sets
$R\cup\{\rt{ab|c}\}$, $R\cup\{\rt{ac|b}\}$, $R\cup\{\rt{bc|a}\}$ is
consistent (e.g.\ with the $\mc{O}(|L_R||R|)$-time algorithm
\texttt{BUILD}), and if so, add the respective triple to the closure
$\cl(R)$ of $R$. A further characterization of the closure by means of the
Ahograph is given by Bryant \cite[Cor.\ 3.9]{Bryant97}.
\begin{theorem} 
	For a consistent triple set $R$ we have $\rt{ab|c}\in \cl(R)$
	if and only if 
	there is a subset $\LT\subseteq L_R$ such that the Ahograph $[R,\LT]$
	has exactly two connected components, one containing $a$ and $b$ 		
		and the other containing $c$.
  \label{thm:2cc}
\end{theorem}

We complete this section with a last result for later reference. 

\begin{lemma}
	Let $R$ be consistent and $R'\subseteq R$. Then	$R'\subseteq \cl(R\setminus  R')$ 
	if and only if $\cl(R\setminus  R') = \cl(R)$.
	In particular, if $\cl(R\setminus  R') = \cl(R)$, then $\cl(R\setminus \{r\}) = \cl(R)$ 
	for any triple $r\in R'$.
	\label{lem:identCl}
\end{lemma}
\begin{proof}
	If	$R'\subseteq \cl(R\setminus  R')$, then clearly 
	$\cl(R\setminus  R') = \cl(R\setminus  R')\cup R'$. 
	Therefore, 
	$\cl(R\setminus  R') =	\cl(\cl(R\setminus  R')) = \cl(\cl(R\setminus  R')\cup R')$.
	Theorem 3.1(8) in \cite{Bryant97} states that $\cl(\cl(A)\cup B) = \cl(A\cup B)$. 
	Hence, $\cl(\cl(R\setminus R')\cup R') = 	\cl((R\setminus R')\cup R') = \cl(R)$.
	Conversely, if $\cl(R\setminus  R') = \cl(R)$, then $R'\subseteq R \subseteq \cl(R)$ implies that
	$R' \subseteq \cl(R\setminus  R')$.

	Now let $R'\subseteq R$,  $r\in R'$ and assume that $\cl(R\setminus  R') = \cl(R)$. 
	Since $R\setminus R'\subseteq R\setminus \{r\}$, we have 
	$\cl(R) =  \cl(R\setminus R')  \subseteq  \cl(R\setminus \{r\}) \subseteq \cl(R)$.  
	Thus, $\cl(R\setminus \{r\}) = \cl(R)$.
\end{proof}

\section{Representative Triple Sets}
\label{sec:repT}

The closure $\cl(R)$ provides all information of further triples that 
are implied by a consistent triple set $R$. Nevertheless, there might be subsets $R'\subseteq R$
that provide the same information, that is, $\cl(R') = \cl(R)$.
See Figure \ref{fig:exmpl} for an example.

\begin{definition}	
	Let $R$ be a consistent triple set. A set $R'\subseteq R$ is \emph{representative}
	for $R$ if $\cl(R) = \cl(R')$. 
	The set $\SC(R)$\footnote{$\SC$ stands for ``same closure''} comprises all representative triple sets of $R$.
	Moreover, we put
	 \[\minimal(\SC(R))\coloneqq \{R'\in \SC(R) \colon R' \text{ is minimal w.r.t.\ inclusion }\}\]
	and 
	 \[\minimum(\SC(R))\coloneqq \{R'\in \SC(R) \colon |R'|\leq|R''| \text{ for any } R''\in \SC(R)\}.\,\]	
\end{definition}

It is easy to see that $\minimum(\SC(R))\subseteq \minimal(\SC(R))$. As we
shall see later, even $\minimum(\SC(R)) = \minimal(\SC(R))$ is satisfied.
In order to investigate the sets $\minimum(\SC(R))$ and $\minimal(\SC(R))$
in more detail, we utilize the Ahograph and, in particular, Theorem
\ref{thm:2cc}. Note, Theorem \ref{thm:2cc} implies that $\rt{ab|c} \in
\cl(R)$ if and only if there is a subset $\LT\subseteq L_R$ such that
$[R,\LT]$ has exactly two connected components $A$ and $B$, one containing
$a,b$ and the other $c$. These two connected components will play a major
role in the proof for matroid properties. Since there might be several
subsets $\LT$ of $L_R$ that satisfy the properties of Theorem \ref{thm:2cc}
for a given triple $\rt{ab|c}$, we collect the respective connected
components $A$ and $B$ in the set $\LrR{\rt{ab|c}}{R}$.

\begin{definition}
	Let $R$ be a consistent triple set and $\rt{ab|c}$ a triple with $a,b,c\in L_R$. 
	The set  \[\LrR{\rt{ab|c}}{R}\]  comprises all sets $\{A,B\}$ for which 
	$A,B\subseteq L_R$ and $[R,A\cup B]$
	has exactly two connected components $A$ and  $B$, 
	one containing $a$ and $b$ and the other containing $c$.
	\label{def:ltr}
\end{definition}

We emphasize that we do not assume that $\rt{ab|c}\in R$ in
Definition \ref{def:ltr}. The following lemma is an immediate consequence of Definition
\ref{def:ltr} and Theorem \ref{thm:2cc}. 

\begin{lemma}
	Let $R$ be a consistent triple set. Then, 
	\[\LrR{\rt{ab|c}}{R}\neq \emptyset \text{ if and only if } \rt{ab|c}\in \cl(R).\,\] 
	\label{lem:nonempty}
\end{lemma}
In what follows, we show that elements $\{\Amax,\Bmax\} \in
\LrR{\rt{ab|c}}{R}$ with $|\Amax \cup \Bmax| \geq |A \cup B|$ for all
$\{A,B\} \in \LrR{\rt{ab|c}}{R}$ are unique in $\LrR{\rt{ab|c}}{R}$ and
that $A$ must be a subset of $\Amax$ (resp. $\Bmax$) while $B$ is a subset
of $\Bmax$ (resp. $\Amax$). In other words, the Ahograph $[R,A\cup B]$ must
be a subgraph of $[R,\Amax\cup \Bmax]$, where one of the two connected
components of $[R,A\cup B]$ is entirely contained in $\Amax$ and the other
in $\Bmax$. To this end, we start with the following lemma.
\begin{lemma}
	Let $R$ be a consistent triple set and $\rt{ab|c}, \rt{a'b'|c'} \in \cl(R)$.
	Assume that $\{A,B\} \in \LrR{\rt{ab|c}}{R}$ and 
	$\{A',B'\} \in \LrR{\rt{a'b'|c'}}{R}$.
	If $A\cap A' \neq \emptyset$ and  $B\cap B' \neq \emptyset$, 
	then $\{A\cup A',B\cup B'\} \in \LrR{\rt{ab|c}}{R}\cap\LrR{\rt{a'b'|c'}}{R}$. 
	\label{lem:intersection-union}
\end{lemma}
\begin{proof}
	Let $R$ be consistent and $\rt{ab|c}, \rt{a'b'|c'} \in \cl(R)$. 
	By Lemma \ref{lem:nonempty}, there are 
	$\{A,B\} \in \LrR{\rt{ab|c}}{R}$ and 
	$\{A',B'\} \in \LrR{\rt{a'b'|c'}}{R}$.
	Assume that $A\cap A' \neq \emptyset$ and  $B\cap B' \neq \emptyset$ and let  
   $\LT' = A\cup A' \cup B \cup B'$.

	By Lemma \ref{lem:subgraph},
	both $[R,A\cup B]$ and  $[R,A'\cup B']$ are subgraphs 
	of $[R,\LT']$. Moreover, by definition, $[R,A\cup B]$ and  $[R,A'\cup B']$ have two connected components $A,B$, resp., $A',B'$.
	Furthermore, since $A\cap A' \neq \emptyset$ and  $B\cap B' \neq \emptyset$
	we can apply Lemma \ref{lem:cc} and conclude that 
	 the induced subgraphs $\langle A\cup A'\rangle_{[R,\LT']}$ and 
	$\langle B\cup B'\rangle_{[R,\LT']}$ 	form connected subgraphs in $[R,\LT']$. 
	Moreover, since $R$ is consistent, Theorem \ref{thm:ahograph} implies that
	$[R,\LT']$ cannot be connected. Hence, $A\cup A'$ and 
	$B\cup B'$ must be the connected components in  $[R,\LT']$, 
	still one containing $a$ and $b$ (resp.\ $a',b'$) and the other $c$ (resp.\ $c'$)
	and therefore, $\{A\cup A',B\cup B'\} \in \LrR{\rt{ab|c}}{R}$
	(resp.\  $\{A\cup A',B\cup B'\} \in\LrR{\rt{a'b'|c'}}{R}$).
\end{proof}

\begin{figure}[tbp]
  \begin{center}
    \includegraphics[width=.9\textwidth]{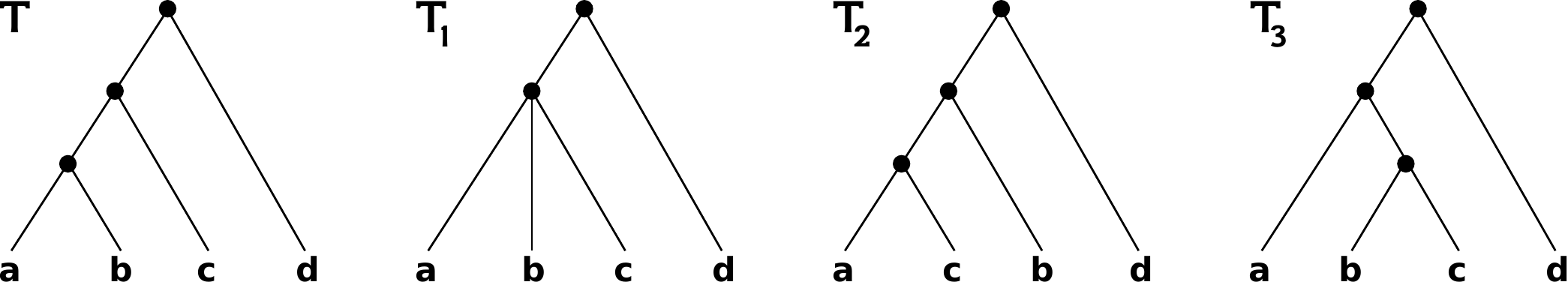}
  \end{center}
	\caption{
		Given the set $R = \{\rt{ab|c},\rt{ac|d},\rt{bc|d}\}$,  		
		there is only one tree $T$ that displays $R$ (shown left). 
		Thus,	$\cl(R) = \mc{R}(T) = \{\rt{ab|c},\rt{ac|d},\rt{bc|d},\rt{ab|d}\}$. 
		The subsets $R_1 = \{\rt{ab|c},\rt{ac|d}\}$ and 
		$R_2 =\{\rt{ab|c}, \rt{bc|d}\}$ are representative triple sets for $R$. 
		In particular, both $R_1$ and $R_2$ are minimal and have minimum size. 
		However, not all subsets of $R$ with size two are representative. 
		By way of example consider $R_3 = \{\rt{ac|d},\rt{bc|d}\}$. 
		Although $T$ displays $R_3$, there are three further trees $T_1,T_2$ and $T_3$ that display $R_3$  as well. 		
		Thus, $\cl(R_3) = \mc{R}(T) \cap \mc{R}(T_1) \cap \mc{R}(T_2) \cap \mc{R}(T_3)=  \{\rt{ac|d},\rt{bc|d},\rt{ab|d}\} \neq \cl(R)$.
    }
	\label{fig:exmpl}
\end{figure}

\begin{lemma}
	Let $R$ be a consistent triple set with $\rt{ab|c} \in \cl(R)$.
	Let $\{\Amax,\Bmax\} \in \LrR{\rt{ab|c}}{R}$ such that 
	$|\Amax\cup \Bmax| \geq 	|A\cup B|$ for all 
	$\{A,B\} \in \LrR{\rt{ab|c}}{R}$. 
	Then,  either $A\subseteq \Amax$ and $B\subseteq \Bmax$
	or $B\subseteq \Amax$ and $A\subseteq \Bmax$.

	Moreover, the element $\{\Amax,\Bmax\}$ in $\LrR{\rt{ab|c}}{R}$ with
		$|\Amax\cup \Bmax| \geq 	|A\cup B|$ for all 
	$\{A,B\} \in \LrR{\rt{ab|c}}{R}$ is unique.
	\label{lem:max1}
\end{lemma}
\begin{proof}
	Let $R$ be a consistent triple set and $\rt{ab|c} \in \cl(R)$.
	By Lemma \ref{lem:nonempty}, the set $\LrR{\rt{ab|c}}{R}$ is not empty, and thus, there is an element  
	 $\{\Amax,\Bmax\} \in \LrR{\rt{ab|c}}{R}$ such that $|\Amax\cup \Bmax| \geq 	|A\cup B|$
	for all 	$\{A,B\} \in \LrR{\rt{ab|c}}{R}$.
	W.l.o.g.\ assume that 	
	$a,b \in A$ and $c\in B$ for some $\{A,B\} \in \LrR{\rt{ab|c}}{R}$. 
	There are two cases, either $a,b\in \Amax$ and $c\in \Bmax$ or, 
	$c\in \Amax$ and $a,b\in \Bmax$. 
	Let us first assume that $a,b\in \Amax$ and $c\in \Bmax$.
	Thus,  $A\cap \Amax \neq \emptyset$ and  $B\cap \Bmax \neq \emptyset$. 
	Lemma \ref{lem:intersection-union} implies that 
	$\{A\cup \Amax, B\cup \Bmax\}\in 	\LrR{\rt{ab|c}}{R}$  and, 
	by choice of $\Amax$ and $\Bmax$, 
	$|\Amax\cup \Bmax| \geq |\LT'|$ where $\LT' = \Amax\cup A\cup \Bmax \cup B$.

	We continue to show that $A\subseteq \Amax$ and $B\subseteq \Bmax$. 
	Since $\{A\cup \Amax, B\cup \Bmax\}\in 	\LrR{\rt{ab|c}}{R}$  we can conclude that 
	$(A\cup \Amax) \cap  (B\cup \Bmax) = \emptyset$. 
	Now, assume for contradiction that $A\not\subseteq \Amax$. 
	Thus, there is an $x\in A\setminus (\Amax \cup \Bmax)$
	and therefore, $\Amax \cup \Bmax \subsetneq (A\cup \Amax) \cup  (B\cup \Bmax)$. 
	Hence, $|(A\cup \Amax)\cup (B\cup\Bmax)|  > |\Amax\cup \Bmax|$; a contradiction. 
   Analogously,  $B\subseteq \Bmax$.

	The latter arguments immediately imply that 
	for any 	$\{\Amax_1,\Bmax_1\}, \{\Amax_2,\Bmax_2\}  \in  \LrR{\rt{ab|c}}{R}$
	with $|\Amax_1\cup \Bmax_1| = |\Amax_2\cup \Bmax_2| \geq 	|A\cup B|$
	for all $\{A,B\}\in  \LrR{\rt{ab|c}}{R}$ it must hold 
	$\{\Amax_1,\Bmax_1\} = \{\Amax_2,\Bmax_2\}$. 
\end{proof}

For our results it will be convenient to explicitly name the 
unique element $\{\Amax,\Bmax\}$ that has maximum cardinality 
$|\Amax\cup \Bmax|$ in $\LrR{\rt{ab|c}}{R}$ as defined  next.

\begin{definition}
	Let $R$ be a consistent triple set with $\rt{ab|c} \in \cl(R)$.
	Then,  \[\lmax{\rt{ab|c}}{R}\] denotes the unique element 
	$\{\Amax,\Bmax\} \in \LrR{\rt{ab|c}}{R}$ for which
	$|\Amax\cup \Bmax| \geq 	|A\cup B|$ for all 
	$\{A,B\} \in \LrR{\rt{ab|c}}{R}$. 
	
	Moreover, for a subset $R' \subseteq \cl(R)$ we set
	\[\Lmax{R'}{R}\coloneqq \bigcup_{\rt{ab|c}\in R'} \{\lmax{\rt{ab|c}}{R}\}.\,\] 
\end{definition}

It is easy to verify that $|\Lmax{R}{R}|\leq |R|$. In what follows, we will show that
for any consistent set $R$ the sets $\Lmax{R}{R}, \Lmax{\cl(R)}{R}$ and $\Lmax{\cl(R)}{\cl(R)}$
are identical. This in turn is used to show that $\Lmax{R}{R} = \Lmax{R'}{R'}$
 whenever $\cl(R')=\cl(R)$ for some $R'\subseteq R$.
In particular, 
the elements in  $\Lmax{R}{R}$ and $\Lmax{R'}{R'}$  are identical w.r.t.\ to a given triple $\rt{ab|c}\in \cl(R)$, that is, 
 $\lmax{\rt{ab|c}}{R} = \lmax{\rt{ab|c}}{R'}$ for any $\rt{ab|c}\in \cl(R)$.
Hence, if $\lmax{\rt{ab|c}}{R} = \{A,B\} $ and $\lmax{\rt{ab|c}}{R'}=\{A',B'\}$, 
then $[R,A\cup B]$  and $[R',A'\cup B']$ have the same two connected components. 
Note, the latter does not imply that the Ahographs 
$[R,A\cup B]$  and $[R',A'\cup B']$ 
are isomorphic. 
We refer to Figure \ref{fig:Wexmpl} for an
illustrative example.
To establish these results we provide first the following lemma.

\begin{lemma}
	Let $R$ be a consistent triple set.
	Assume that there are distinct $\{A',B'\}, \{A,B\} \in \Lmax{R}{R}$.
	If $A'\cap(A\cup B)\neq \emptyset$ and $B'\cap(A\cup B)\neq \emptyset$,
 	then $A'\cup B'\subseteq A$ or  $A'\cup B'\subseteq B$. 
	\label{lem:intersection-subset}
\end{lemma}
\begin{proof}
	If  $\{A,B\},\{A',B'\}$ are distinct elements of $\Lmax{R}{R}$, then there are distinct triples 
	$\rt{ab|c}$ and $\rt{a'b'|c'}$  in $R$ such that $\lmax{\rt{ab|c}}{R}= \{A,B\}$
	and $\lmax{\rt{a'b'|c'}}{R}= \{A',B'\}$.
	Assume that $A'\cap(A\cup B)\neq \emptyset$ and $B'\cap(A\cup B)\neq \emptyset$.

	First consider the case $A'\cap A\neq \emptyset$ \emph{and} $A'\cap B\neq \emptyset$. 
   Lemma \ref{lem:cc}  implies that 
	the induced subgraph $\langle A'\cup A\cup B\rangle$
 	of $[R, A\cup B\cup A'\cup B']$ is connected.
	Since $B'\cap(A\cup B)\neq \emptyset$ and $B'$ is a connected component in $[R, A'\cup B']$, 
	we can apply Lemma \ref{lem:subgraph} and \ref{lem:cc}  and conclude that 
	$[R, A\cup B\cup A'\cup B']$ is a connected graph; a contradiction to Theorem \ref{thm:2cc}.
	Hence, the case $A'\cap A\neq \emptyset$ \emph{and} $A'\cap B\neq \emptyset$ cannot occur. 
	Similarly, $B'\cap A\neq \emptyset$ \emph{and} $B'\cap B\neq \emptyset$	is not possible.
	Thus, we have either $A'\cap A\neq \emptyset$ or $A'\cap B\neq \emptyset$ as well as,	
	either $B'\cap A\neq \emptyset$ or $B'\cap B\neq \emptyset$.

	First assume that $A'\cap A\neq \emptyset$ and $B'\cap B\neq \emptyset$.
	By Lemma \ref{lem:intersection-union}, $\{A\cup A',B\cup B'\} \in \LrR{\rt{ab|c}}{R} \cap\LrR{\rt{a'b'|c'}}{R}$.
	Thus, by Lemma \ref{lem:max1} we have, on the one hand,
	$A\cup A'\subseteq A$ and $B\cup B'\subseteq B$
	and, one the other hand, $A\cup A'\subseteq A'$ and $B\cup B'\subseteq B'$.
	Hence, $A=A'$ and $B=B'$; a contradiction since we assumed that $\{A,B\}$ and $\{A',B'\}$
	are distinct. 
	Thus, the case $A'\cap A\neq \emptyset$ and $B'\cap B\neq \emptyset$ cannot occur. 
	Similarly, the case $B'\cap A\neq \emptyset$ and $A'\cap B\neq \emptyset$ is impossible.

	Therefore, we are left with two exclusive cases: (1) $A'\cap A\neq \emptyset$ and $B'\cap A\neq \emptyset$
	or (2) $A'\cap B\neq \emptyset$ and $B'\cap B\neq \emptyset$.
	Let us assume case (1) $A'\cap A\neq \emptyset$ and $B'\cap A\neq \emptyset$.
	Repeated application of Lemma \ref{lem:cc} shows that the induced subgraph $\langle A\cup A'\cup B'\rangle$
	of $[R, A\cup B\cup A'\cup B']$ is connected.
	Since $R$ is consistent, Theorem \ref{thm:2cc} implies that the graph $[R, A\cup B\cup A'\cup B']$ must be disconnected. 
	Hence, $[R, A\cup B\cup A'\cup B']$ has as connected components $A\cup A'\cup B'$
	and $B$. Therefore, $\{A\cup A'\cup B', B\} \in \LrR{\rt{ab|c}}{R}$.
	Now it must hold that  $A'\cup B'\subseteq A$  as otherwise
	$|A\cup A'\cup B'\cup B|>|A\cup B| = \lmax{\rt{ab|c}}{R}$ would yield a contradiction. 
	In case (2) it is shown analogously that $A'\cup B'\subseteq B$.
\end{proof}

Lemma \ref{lem:intersection-subset} immediately implies the following
\begin{cor}
	Let $R$ be a consistent triple set.
	Assume that there are distinct $\{A,B\}, \{A',B'\} \in \Lmax{R}{R}$. 
	If $A\cap A'\neq \emptyset$, then $B\cap B'=\emptyset$. 
	\label{cor:Bempty}
\end{cor}
\begin{proof}
		Assume that $A\cap A'\neq \emptyset$ and $B\cap B'\neq \emptyset$. 
	Hence, $A'\cap(A\cup B)\neq \emptyset$ and $B'\cap(A\cup B)\neq \emptyset$.
	Lemma \ref{lem:intersection-subset} implies that
	$A'\cup B'\subseteq A$ or  $A'\cup B'\subseteq B$.
	W.l.o.g.\ assume that $A'\cup B'\subseteq A$.  
	Analogously, Lemma \ref{lem:intersection-subset} implies that 
	$A\cup B\subseteq A'$ or  $A\cup B\subseteq B'$.
	Now,  $A'\cup B'\subseteq A$ and $A\cup B\subseteq A'$    
	would imply that $B' \subseteq A'$; a contradiction. 
	Furthermore, if $A\cup B\subseteq B'$, then
  $A'\cup B'\subseteq A $ 
	would imply that $B \subseteq A$; again a contradiction. 
\end{proof}

\begin{figure}[t]
  \begin{center}
    \includegraphics[width=.8\textwidth]{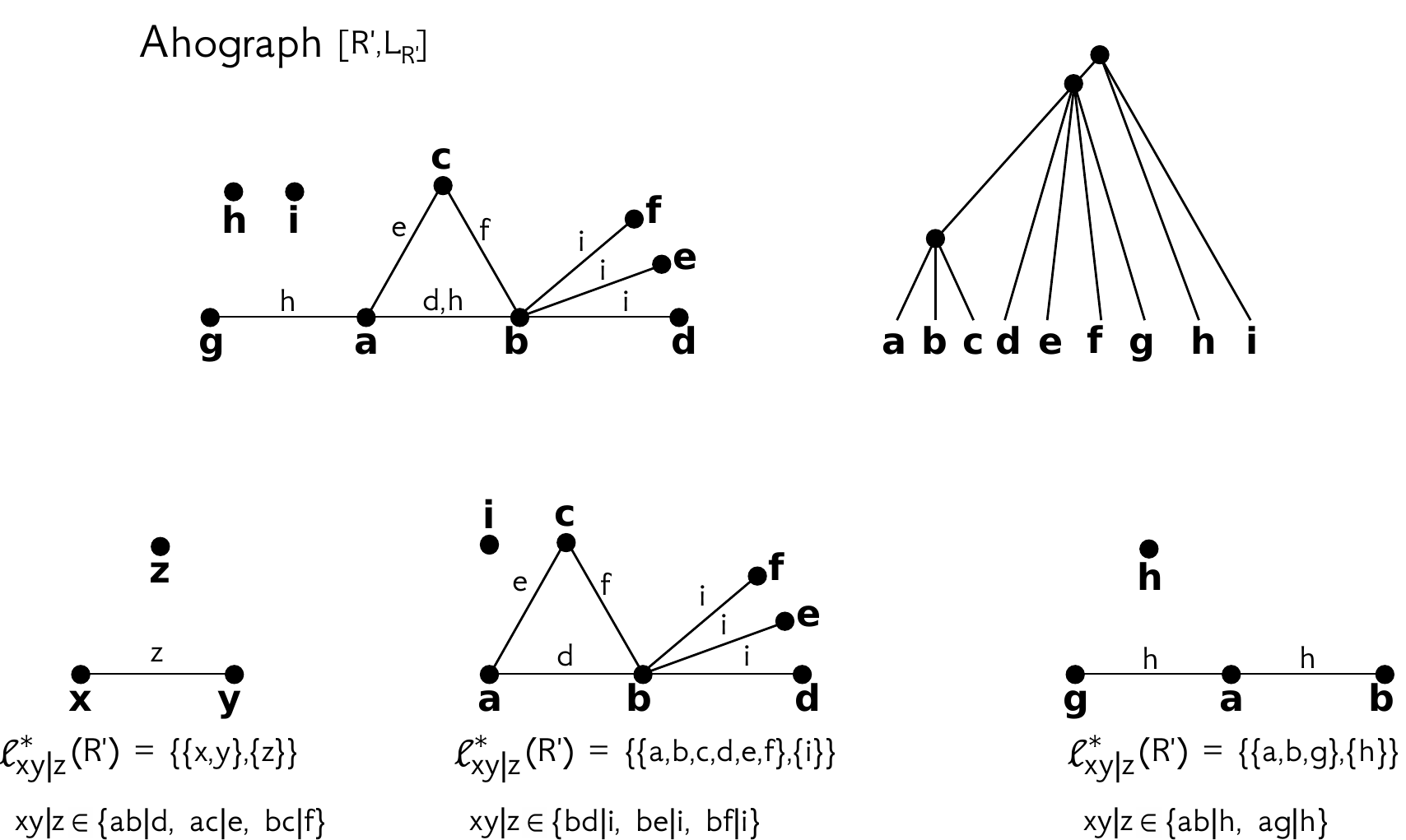}
  \end{center}
	\caption{
		Consider the triple set 
		$R = \{\rt{ab|d},\rt{ab|h},\rt{ac|e},\rt{ag|h},
				 \rt{bc|f},\rt{bc|i},\rt{bd|i},\rt{be|i},\rt{bf|i},\rt{bg|h}\}$
		and let $R'=R\setminus \{\rt{bc|i},\rt{bg|h}\}$. In this example, $L_{R'} = L_{R}$.
		Clearly, any tree that display $R'$ and thus, 
		$\{\rt{bc|f}, \rt{bf|i}\}$, resp., $\{\rt{ab|h}, \rt{ag|h}\}$, 
		must  also display $\rt{bc|i}$, resp., $\rt{bg|h}$ \cite{Dekker86}.
		Thus, $\{\rt{bc|i}, \rt{bg|h}\} \subseteq \cl(R') = \cl(R\setminus \{\rt{bc|i},\rt{bg|h}\})$.
		Lemma \ref{lem:identCl} implies that $\cl(R') = \cl(R)$ and therefore, 
		$R'\in\SC(R)$. 
		The Ahograph $[R',L_{R'}]$ and a tree $T$ that displays $R'$ is shown
		on the top of this figure. In $[R',L_{R'}]$ each edge $(x,y)$ is labeled
		with $z$ that corresponds to the respective triple $\rt{xy|z}$ that 
		supports the edge $(x,y)$. \newline
		Since $R\subseteq \cl(R) = \cl(R')$,  the tree $T$ also displays $R$.  The respective maximal elements 
		$\lmax{\rt{xy|z}}{R'} =\{A,B\} \in \Lmax{R'}{R'}$ with corresponding 
		Ahographs $[R',A\cup B]$ are depicted below. 	
		It is easy to verify that each triple $\rt{xy|z} \in R'$ is a bridge 	
		in the respective Ahograph $[R',A\cup B]$ and hence, $R'$ is minimal (cf.\ Lemma \ref{lem:bridge}).
		By Theorem \ref{thm:min=MIN},  $R'$ has also minimum cardinality. 
		Note, $\cl(R')=\cl(R)$ and Theorem \ref{thm:equals} imply that 
		$\Lmax{R'}{R'} = \Lmax{\cl(R)}{\cl(R)} = \Lmax{R}{R}$. 
		In this example,
 		$\Lmax{R}{R} = \big\{ \big\{\{a,b\},\{d\}\big\}, \big\{\{a,c\},\{e\}\big\}, \big\{\{b,c\},\{f\}\big\}, 
							\big\{\{a,b,c,d,e,f\},\{i\}\big\}, \big\{\{a,b,g\},\{h\}\big\}\big\}$.
		In order to determine $\cl(R)$ it suffices to add for 
		each $\{A,B\} \in \Lmax{R}{R}$ all triples $\rt{xy|z}$ with $x,y\in A, z\in B$
		or $z\in A, x,y\in B$	to $\cl(R)$ (cf.\, Thm.\, \ref{thm:cl-R}).
		Finally, application of Theorem \ref{thm:cor:boundary} shows that $R'$ does not 
	   identify $T$, since $9 = B(T) > |R'|=8$ and thus, $\cl(R')\neq \mathcal{R}(T)$. Moreover, since 
		$\cl(R) = \cl(R')$ neither $R$ identifies $T$.
    }
	\label{fig:Wexmpl}
\end{figure}

\begin{theorem}
	For any consistent triple set $R$ it holds that	
	\[\Lmax{R}{R} = \Lmax{\cl(R)}{R} = \Lmax{\cl(R)}{\cl(R)}.\,\]
	\label{thm:equals}
\end{theorem}
\begin{proof}
	We start with showing $\Lmax{R}{R} = \Lmax{\cl(R)}{R}$. 
	Since $R\subseteq \cl(R)$, we also have 
	$\Lmax{R}{R} \subseteq \Lmax{\cl(R)}{R}$. 

	To see that $\Lmax{\cl(R)}{R} \subseteq \Lmax{R}{R}$, let
	 $\{A,B\}\in \Lmax{\cl(R)}{R}$. 
	Hence, there is a triple $\rt{ab|c}\in \cl(R)$
	with $\lmax{\rt{ab|c}}{R} = \{A,B\} \in \LrR{\rt{ab|c}}{R}$.
	By definition, $[R,A\cup B]$ has  two connected components and at least
	one contains 2 or more vertices. First, assume for contradiction that 
	there is no
	triple $\rt{a'b'|c'} \in R$ with 
	$\{A,B\} \in \LrR{\rt{a'b'|c'}}{R}$. Since
	$|A|>1$ or $|B|>1$ and $A,B$ are connected components in $[R,A\cup B]$, all
	edges $(x,y)$ within $\langle A \rangle_{[R,A\cup B]}$ and $\langle B \rangle_{[R,A\cup B]}$ are,
	therefore, provided by triples $\rt{xy|z}\in R$ such that either $x,y,z\in A$ or
	$x,y,z\in B$. But then $[R,A]$ or $[R,B]$  is connected; contradicting
	Theorem \ref{thm:ahograph}. Thus, there must be a triple $\rt{a'b'|c'} \in R$
	with $\{A,B\} \in \LrR{\rt{a'b'|c'}}{R}$.
	We continue with showing that $\lmax{\rt{a'b'|c'}}{R} = \{A,B\}$. If
	this was not the case, then there is an $\lmax{\rt{a'b'|c'}}{R} =
	\{\Amax,\Bmax\}$ with $|\Amax \cup \Bmax| > |A\cup B|$. Since
	$\lmax{\rt{ab|c}}{R} = \{A,B\} \in \LrR{\rt{ab|c}}{R}$, one of $A$ and
	$B$ is containing $a,b$ and the other $c$. Moreover, Lemma
	\ref{lem:max1} implies that $A\subseteq \Amax$ and $B\subseteq \Bmax$
	or $B\subseteq \Amax$ and $A\subseteq \Bmax$. Taken the latter two
	arguments together, one of $\Amax$ and $\Bmax$ is containing $a,b$ and
	the other $c$. Therefore, $\{\Amax,\Bmax\} \in \LrR{\rt{ab|c}}{R}$.
	However, since $|\Amax \cup \Bmax| > |A\cup B|$, we have
	$\lmax{\rt{ab|c}}{R} \neq \{A,B\}$; a contradiction to the assumption
	$\lmax{\rt{ab|c}}{R} = \{A,B\}$. Therefore, $\lmax{\rt{a'b'|c'}}{R} =
	\{A,B\}$ and thus, $\{A,B\}\in \Lmax{R}{R}$. Hence, $\Lmax{R}{R} =
	\Lmax{\cl(R)}{R}$.

	Now we show that $\Lmax{\cl(R)}{R} = \Lmax{\cl(R)}{\cl(R)}$.
	Let $\{A,B\}  = \lmax{\rt{ab|c}}{R} \in \Lmax{\cl(R)}{R}$. 
	Since $R\subseteq \cl(R)$ and by Lemma \ref{lem:subgraph}, 
	$[R,A\cup B]$ is a subgraph of $[\cl(R),A\cup B]$.
	Since $\{A,B\}  = \lmax{\rt{ab|c}}{R}$, 
	the graph $[R,A\cup B]$ has exactly two connected components $A$ and $B$. 
	Thus, $[\cl(R),A\cup B]$ has at most two connected components. 
	Still, the induced subgraphs $\langle A  \rangle$  and $\langle B \rangle$ 
	of $[\cl(R),A\cup B]$ are connected. 
	However, since $\cl(R)$ is consistent we can apply 
	Theorem \ref{thm:ahograph} and conclude that 
	$[\cl(R),A\cup B]$ must be disconnected, and thus, has as connected components	
	$A$ and $B$. Therefore, $\{A,B\}\in \LrR{\rt{ab|c}}{\cl(R)}$. 
	Assume now for contradiction that 
	$\lmax{\rt{ab|c}}{\cl(R)} = \{\Amax,\Bmax\} \neq \{A,B\}$.
   By Lemma \ref{lem:max1}, 
	 $|\Amax \cup \Bmax| >  |A\cup B|$ and either
	$A\subseteq \Amax$ and $B\subseteq \Bmax$
	or $B\subseteq \Amax$ and $A\subseteq \Bmax$. 
	W.l.o.g.\ assume that $A\subseteq \Amax$ and $B\subseteq \Bmax$
	and $a,b\in A$, $c\in B$. 
	Hence, there is a vertex $d\in (\Amax\cup \Bmax)\setminus (A\cup B)$.
	If $d\in \Amax$, then Theorem \ref{thm:2cc} and $a\in \Amax$, $c\in \Bmax$ imply that 
	$\rt{ad|c} \in \cl(\cl(R))=\cl(R)$.
	Again, Theorem \ref{thm:2cc} implies that there is 
	a subset $\LT\subseteq L_R$ such that $[R,\LT]$
	has exactly two connected components $A'$, $B'$ with $a,d\in A'$ and $c\in B'$. 
	Thus, $\{A',B'\} \in \LrR{\rt{ad|c}}{R}$. Recap that $\{A,B\} \in \LrR{\rt{ab|c}}{R}$.
	Since $A\cap A'\neq \emptyset$ and $B\cap B'\neq \emptyset$
	we can apply Lemma \ref{lem:intersection-union} and conclude that 
	 $\{A\cup A',B\cup B'\} \in \LrR{\rt{ab|c}}{R}$. 
	However, since $d\in A'\setminus A$ it holds that 
	$|A\cup A'\cup B\cup B'| > |A\cup B|$; a contradiction to $\{A,B\}= \lmax{\rt{ab|c}}{R}$.
	By similar arguments one derives a contradiction if $d\in \Bmax$.
	Thus, $\lmax{\rt{ab|c}}{\cl(R)} = \{\Amax,\Bmax\} = \{A,B\}$
	and therefore, $\{A,B\} \in \Lmax{\cl(R)}{\cl(R)}$. Hence, $\Lmax{\cl(R)}{R} \subseteq \Lmax{\cl(R)}{\cl(R)}$.

	Finally, let  $\{A,B\}  = \lmax{\rt{ab|c}}{\cl(R)} \in \Lmax{\cl(R)}{\cl(R)}$.
	By Lemma \ref{lem:nonempty} and since $\rt{ab|c}\in \cl(R)$, we have
	$\LrR{\rt{ab|c}}{R}\neq \emptyset$. Thus, there is a maximal element
	$\lmax{\rt{ab|c}}{R}  = \{\Amax,\Bmax\} \in \Lmax{\cl(R)}{R} \subseteq \Lmax{\cl(R)}{\cl(R)}$. 
	Assume that $\{A,B\}$ and $\{\Amax,\Bmax\}$ are distinct. 
	Since both $\{A,B\}$ and $\{\Amax,\Bmax\}$ are contained in $\Lmax{\cl(R)}{\cl(R)}$
	and $A\cap (\Amax\cup\Bmax)\neq \emptyset$,  $B\cap (\Amax\cup\Bmax)\neq \emptyset$
	we can apply Lemma \ref{lem:intersection-subset} and conclude that
	$A\cup B\subseteq \Amax$ or $A\cup B \subseteq \Bmax$. 
	If $A\cup B\subseteq \Amax$, then $a,b,c\in \Amax$; a contradiction, 
	since one of $\Amax$ and $\Bmax$ contains $a,b$ and the other $c$. 
	Analogously, $A\cup B \subseteq \Bmax$ cannot occur. 
		Hence, $\{A,B\}$ and $\{\Amax,\Bmax\}$ must be equal and therefore, 
   $\lmax{\rt{ab|c}}{R}  = \lmax{\rt{ab|c}}{\cl(R)} = \{A,B\} \in \Lmax{\cl(R)}{R}$.
	Thus, $\Lmax{\cl(R)}{R} = \Lmax{\cl(R)}{\cl(R)}$.			
\end{proof}

\begin{theorem}
	Let $R$ be a consistent triple set. 
   If $R'\in \SC(R)$, then $\Lmax{R}{R} = \Lmax{R'}{R'}$.
	In particular, for every $\rt{ab|c}\in cl(R)$ and  $R'\in \SC(R)$, 
   it holds that $\lmax{\rt{ab|c}}{R'} =  \lmax{\rt{ab|c}}{R}$.
	\label{thm:repres}
\end{theorem}
\begin{proof}
	Let $R'\in \SC(R)$ and thus, $\cl(R') = \cl(R)$. Therefore,
	\[\Lmax{R}{R} \stackrel{\textrm{Thm.\ \ref
	{thm:equals}}}{=}\Lmax{\cl(R)}{\cl(R)}
	=\Lmax{\cl(R')}{\cl(R')}\stackrel{\textrm{Thm.\ \ref {thm:equals}}}{=}
	\Lmax{R'}{R'}.\,\] Now let $\lmax{\rt{ab|c}}{R} = \{\Amax,\Bmax\} \in
	\Lmax{\cl(R)}{R}$. Note, Theorem \ref{thm:equals} implies that
	$\Lmax{\cl(R)}{R}=\Lmax{R}{R}$ and therefore, $\{\Amax,\Bmax\} \in
	\Lmax{R}{R} = \Lmax{R'}{R'}$. Since $\rt{ab|c}\in \cl(R) = \cl(R')$ we
	can apply Lemma \ref{lem:nonempty} and conclude that
	$\LrR{\rt{ab|c}}{R'}\neq \emptyset$. Let $\lmax{\rt{ab|c}}{R'} =
	\{A,B\} \in \Lmax{\cl(R')}{R'} \stackrel{\textrm{Thm.\ \ref
	{thm:equals}}}{=} \Lmax{R'}{R'}$. Thus, both $\{\Amax,\Bmax\}$ and
	$\{A,B\} $ are contained in $\Lmax{R'}{R'}$ and $A\cap
	(\Amax\cup\Bmax)\neq \emptyset$, $B\cap (\Amax\cup\Bmax)\neq
	\emptyset$. Now we can argue analogously as in the last part of the
	proof of Theorem \ref{lem:nonempty} to conclude that $\{A,B\} =
	\{\Amax,\Bmax\}$ which implies that $\lmax{\rt{ab|c}}{R'} =
	\lmax{\rt{ab|c}}{R}$. 
\end{proof}

\section{The Matroid Structure of Minimal and Minimum Representative Triple Sets}
\label{sec:min}

By definition, $R'\in \minimal(\SC(R))$  if and only if there is no subset $R''\subsetneq R'$ with $\cl(R'') =\cl(R)$. 
Furthermore, since any minimum representative triple set is, in particular, minimal, we have
 $\minimum(\SC(R)) \subseteq \minimal(\SC(R))$. 
The computation of a minimal representative set $R'$ of $R$ can be done 
in combination with the $\mc{O}(|R||L_R|^4)$ method to compute the closure \cite{BS:95}
in polynomial time
as follows: Set $R'=R$ and as long as there is a triple $r\in \cl(R'\setminus r)$ remove
				$r$ from $R'$. By Lemma \ref{lem:identCl}, removal of $r$ from $R'$
				still preserves $\cl(R) = \cl(R')$.
However, the computational complexity of finding a minimum 
representative set $R'$ of $R$ is still an open problem. 
We show that one can determine  minimum representative sets in polynomial time.
To this end, we give the following

\begin{definition}
	A \emph{matroid} is an ordered pair $(E,\mathbb{F}_E)$ consisting of a finite set $E$ and a 
	collection $\mathbb{F}_E$ of subsets of $E$ having the following three properties:
	\begin{itemize}
	\item[(I1)] $\emptyset\in \mathbb{F}_E$;
	\item[(I2)]	If $I\in \mathbb{F}_E$ and $I'\subseteq I$, then $I'\in \mathbb{F}_E$;
	\item[(I3)]	If $I_1,I_2\in \mathbb{F}_E$ and $|I_1|<|I_2|$, then there is an element 
							$x\in I_2\setminus I_1$ such that $I_1\cup \{x\}\in \mathbb{F}_E$.
	\end{itemize}
\end{definition}

The elements in $\mathbb{F}_E$ are called \emph{independent} in $(E,\mathbb{F}_E)$. 
Maximal independent elements of a matroid are called a basis of $(E,\mathbb{F}_E)$. 
Every matroid  $(E,\mathbb{F}_E)$ is determined by its collection of its bases. 
We refer the reader to \cite{Oxley:11,korte2012combinatorial} for more detailed background on matroid theory. 

In what follows, we show that $\minimal(\SC(R))$ forms the collection of bases of a matroid.
In this case, $\minimum(\SC(R))  = \minimal(\SC(R))$ since 
all basis elements of a matroid have the same cardinality \cite{Oxley:11,korte2012combinatorial}.
A useful characterization  is given by the next result. 
\begin{lemma}[{\cite[Cor.\ 1.2.5]{Oxley:11}}]	
	Let $\mc{B}$ be  a collection of subsets of $E$. 
	Then $\mc{B}$ is the collection of bases of a matroid $(E,\mathbb{F}_E)$ if and only
	if it has the following properties:
	\begin{itemize}
	\item[(B1)] $\mc{B}\neq \emptyset$;
	\item[(B2)] If $B_1,B_2\in \mc{B}$ and $x\in B_1\setminus B_2$, then
					there is an element $y\in B_2\setminus B_1$ such 
					$ (B_1\setminus \{x\}) \cup \{y\} \in \mc{B}$.
	\end{itemize}
	\label{lem:basis}
\end{lemma}

\begin{definition}
In what follows, $(R,\mathbb{F}_R)$ denotes the ordered pair where
\begin{enumerate}	
	\item $R$ is a consistent triple set and 
	\item $\mathbb{F}_R =\{R''\subseteq R'\colon R'\in \minimal(\SC(R))\}$ is the collection of all
			subsets of the minimal representative sets of $R$. 
\end{enumerate} 
\end{definition}

It is easy to see that  $(R,\mathbb{F}_R)$ is an independent system, that is, it satisfies Conditions (I1) and (I2).
Moreover, the collection of bases of $(R,\mathbb{F}_R)$ is the set $\minimal(\SC(R))$. 
We will utilize Lemma \ref{lem:basis} to show $\mc{B} = \minimal(\SC(R))$ satisfies (B1) and (B2). 
To this end, we 
give the notion of ``bridges'' in the Ahograph, that is, triples $\rt{ab|c}$ 
for which the Ahograph $[R\setminus \{\rt{ab|c}\},\LT]$  has more connected components than $[R,\LT]$ . 
As it turns out, elements  $R'\in \minimal(\SC(R))$  are characterized by 
the bridge-property of triples $\rt{ab|c}\in R'$. We first give the following result.

\begin{lemma}
	Let $R$ be a consistent triple set and $R'\in \SC(R)$. 
	Then $R'\in \minimal(\SC(R))$ if and only if 
  	$\cl(R')\neq \cl(R'\setminus \{r\})$ for all $r\in R'$.
	\label{lem:sc-check}
\end{lemma}
\begin{proof}
	Let $R$ be a consistent triple set and $R'\in \SC(R)$ and thus, 
	$\cl(R') = \cl(R)$.
	Clearly, if $\cl(R) = cl(R') = \cl(R'\setminus \{r\})$ for any $r\in R'$, 
	then $R'\notin \minimal(\SC(R))$. Conversely, if 
	$R'\notin \minimal(\SC(R))$, then 
	there is a subset $R''\subsetneq R'$
	with $\cl(R'') = \cl(R)$. Since $R'\in \SC(R)$, it also holds that $\cl(R') = \cl(R) = \cl(R'')$.
	Let $r\in R'\setminus R''$. 
	Since $R''\subseteq R'\setminus \{r\} \subsetneq R'$, we have 
	$\cl(R'')\subseteq \cl(R'\setminus \{r\}) \subseteq \cl(R') = \cl(R'')$
	and therefore, $\cl(R'\setminus \{r\}) = \cl(R')$. 
\end{proof}

\begin{definition}
	Let $R$ be a consistent triple set, $\rt{ab|c}\in R$  
   and $\LT\subseteq L_R$ such that $a,b,c \in \LT$. 
	The triple $\rt{ab|c}$ is called \emph{bridge} in $[R,\LT]$
	if $a,b$ are in different connected components of $[R\setminus \{\rt{ab|c}\},\LT]$.
  \end{definition}
\begin{lemma}
		Let $R$ be a consistent triple set and $R'\in \SC(R)$. 
		Then, $R'\in \minimal(\SC(R))$ if and only if every
		$\rt{ab|c} \in R'$ is a bridge in $[R',A\cup B]$ with
		$\{A,B\} = \lmax{\rt{ab|c}}{R'}$.
		In particular,  $[R'\setminus\{\rt{ab|c}\},A\cup B]$
		must have three connected components $\alpha,\beta,\gamma$
		with $a\in \alpha$, $b\in \beta$ and $c\in \gamma$, 
		that is, either $A=\alpha \cup \beta$ and $B=\gamma$ or
		$B=\alpha \cup \beta$ and $A=\gamma$.
	\label{lem:bridge}
\end{lemma}
\begin{proof}
	Let $R'\in \minimal(\SC(R))$, $\rt{ab|c} \in R'$ and $\lmax{\rt{ab|c}}{R'} = \{A,B\}$. 
	By definition, $[R',A\cup B]$ has exactly two connected components, one containing $a,b$ and
	the other $c$.
	Assume for contradiction that $\rt{ab|c}$ is not a bridge in 
	$[R',A\cup B]$. Thus, $a$ and $b$ are still connected by a walk in 
	$[R'\setminus \{\rt{ab|c}\}, A\cup B]$.
	Note, by Lemma \ref{lem:subgraph} the Ahograph $[R'\setminus \{\rt{ab|c}\}, A\cup B]$ is 
	a subgraph of $[R',A\cup B]$ that differs from $[R',A\cup B]$ only
	by the edge $(a,b)$. Therefore, 
	 $[R'\setminus \{\rt{ab|c}\}, A\cup B]$ still consists of the two connected components 
	$A$ and $B$, one containing $a,b$ and
	the other $c$. Theorem \ref{thm:2cc} implies that 
	$\{\rt{ab|c}\}  \in \cl(R'\setminus \{\rt{ab|c}\})$. 
   Lemma \ref{lem:identCl} implies that $\cl(R'\setminus \{\rt{ab|c}\}) = \cl(R') = \cl(R)$; 
	a contradiction to 
	$R'\in \minimal(\SC(R))$.	

	Conversely, assume that $R'\not\in \minimal(\SC(R))$. Thus, there is some triple $\rt{ab|c}\in R'$	
	such that $\cl(R'\setminus\{\rt{ab|c}\}) = \cl(R)$. Since $R'\in \SC(R)$, we can apply
	Theorem \ref{thm:repres} and conclude that 
	$\lmax{\rt{ab|c}}{R'\setminus\{\rt{ab|c}\}} = \lmax{\rt{ab|c}}{R'}  = \{A,B\}$.
	Thus, $[R'\setminus\{\rt{ab|c}\}, A\cup B]$ has two connected components $A$ and $B$. Therefore, 
	$\rt{ab|c}$ is not a bridge in $[R', A\cup B]$. 

   For the last statement, observe that 
	$R'\in \minimal(\SC(R))$ 	and	$\{A,B\} = \lmax{\rt{ab|c}}{R'}$ implies that
	the graph $[R',A\cup B]$ has exactly two connected components, one containing
	$a,b$ (say $A$) and the other ($B$) contains $c$. 
	Since $\rt{ab|c} \in R'$ is a bridge in $[R',A\cup B]$, 
	$a$ and $b$ are in distinct connected components $\alpha$ and $\beta$ of $[R'\setminus \{\rt{ab|c}\}, A\cup B]$, respectively.
	However, since only the edge $(a,b)$ has been removed from $[R',A\cup B]$
	to obtain $[R'\setminus \{\rt{ab|c}\}, A\cup B]$ it is clear that 
	the set $A$ decomposes into these connected components $\alpha,\beta$, i.e., $A=\alpha\cup \beta$.
	Besides the edge $(a,b)$  no other edge  has been removed or added to 	$[R'\setminus \{\rt{ab|c}\}, A\cup B]$
	and thus, 
	$B = \gamma$ is still a connected component in $[R'\setminus \{\rt{ab|c}\}, A\cup B]$
	with $c\in \gamma$.
\end{proof}

We are now in the position to show that is a matroid.

\begin{theorem}
	If $R$ is a consistent triple set, then $(R,\mathbb{F}_R)$ is a matroid.
	\label{thm:matroid}
\end{theorem}
\begin{proof}
	In order to show that $\RF$ is a matroid, we show that its collection of bases $\mc{B} = \minimal(\SC(R))$
	satisfies the Conditions (B1) and (B2)  of Lemma \ref{lem:basis}. 
	Recall that $(R,\mathbb{F}_R)$ is an independent system with 
	collection of bases $\mc{B} = \minimal(\SC(R))$ and $\minimal(\SC(R))\neq \emptyset$. Thus, 
	Condition (B1) is trivially satisfied. 
	The proof of Condition (B2) consists of several steps (Claim 1 - 5). 

	We fix the notion as follows:
	We assume that $R_1,R_2\in \mc{B}$, 
	$\rt{ab|c}\in R_1\setminus R_2$ and  $\lmax{\rt{ab|c}}{R_1}=\{A,B\} \in \Lmax{R_1}{R_1} $.
	Moreover, we will frequently make use of
	$\Lmax{R_1}{R_1} = \Lmax{R}{R} = \Lmax{R_2}{R_2}$, which is because of $\cl(R_1)=\cl(R)=\cl(R_2)$
	and Theorem \ref{thm:repres}.
	Furthermore, 
	Lemma \ref{lem:bridge} implies that $\rt{ab|c}$ is a bridge in $[R_1,A\cup B]$
	and that $[R_1\setminus\{\rt{ab|c}\},A\cup B]$ decomposes into the connected components
	 $\alpha,\beta,\gamma$	with $a\in \alpha$, $b\in \beta$ and $c\in \gamma$. 
 	 W.l.o.g.\ we will assume that $a,b\in A, c\in B$ and thus,	$A=\alpha\cup \beta$ and $B=\gamma$.

\smallskip
  \begin{description}
  \item[\emph{Claim 1:}] \emph{There exists a triple $\rt{a'b'|c'} \in R_2$ with 
		$a'\in \alpha$, 		$b'\in \beta$ and	$c'\in \gamma$.  }

    \emph{Proof of Claim 1.}  We begin by showing that there is a triple 
	$\rt{a'b'|c'} \in R_2$ such that $a'\in \alpha$, 	$b'\in \beta$ and	$c'\in A\cup B$
	and then show  that $c'\in \gamma$. 

	Assume for contradiction that there is no triple $\rt{a'b'|c'} \in R_2$ such that $a'\in \alpha$, 	$b'\in \beta$ and	$c'\in A\cup B$.
	Hence, there is no edge $(x,y)$ in $[R_2,A\cup B]$ for any $x\in \alpha$ and $y\in \beta$, that is, 
	$\langle A \rangle$ is disconnected in  $[R_2,A\cup B]$. But then 
	$\{A,B\} \notin \Lmax{R_2}{R_2}$; contradicting 
	$\Lmax{R_1}{R_1} = \Lmax{R_2}{R_2}$. 
	Thus there is a triple  $\rt{a'b'|c'} \in R_2$ such that $a'\in \alpha$, 	$b'\in \beta$ and	$c'\in A\cup B$.
	
	We continue to show that $c'\in \gamma$. 
	Assume for contradiction that $c'\notin \gamma$. Since $\gamma=B$, we have $c'\in A$. 
	Let $\lmax{\rt{a'b'|c'}}{R_2} = \{A',B'\}\in \Lmax{R_2}{R_2}$. Note, since $\Lmax{R_1}{R_1} = \Lmax{R_2}{R_2}$ 
	we also have $\{A',B'\} \in \Lmax{R_1}{R_1}$ 
	and hence, 
	the graph $[R_1,A'\cup B']$ has the two connected components $A'$ and $B'$, one containing $a',b'$ and the other $c'$.
	Furthermore, since $a',b',c'\in A$ we have $A\cap A'\neq \emptyset$ and $A\cap B'\neq \emptyset$. 
	Hence, 	$A'\cap(A\cup B)\neq \emptyset$ and $B'\cap(A\cup B)\neq \emptyset$,
	and we can apply Lemma \ref{lem:intersection-subset} to conclude that 
	$A'\cup B'\subseteq A$. Therefore, Lemma \ref{lem:subgraph}  implies that $[R_1,A'\cup B']\subseteq [R_1,A\cup B]$.
	In particular, both $\langle A' \rangle \subseteq \langle A \rangle$ and 
	$\langle B' \rangle \subseteq \langle A \rangle$ are 
	connected subgraphs in  $[R_1,A\cup B]$. 
	Since $c\notin A$ we have $c\notin A'\cup B'$
	and thus, $[R_1\setminus \{\rt{ab|c}\}, A'\cup B'\}] = [R_1,A'\cup B']$. 
	Hence, $\langle A' \rangle$ and $\langle B' \rangle$ remain connected subgraphs in 
   $[R_1\setminus \{\rt{ab|c}\}, A\cup B\}]$. 
	Since $\lmax{\rt{a'b'|c'}}{R_2} = \{A',B'\}$ it holds that either $a',b'\in A'$ and $c'\in B'$
	or $a',b'\in B'$ and $c'\in A'$. 	Assume that $a',b'\in A'$. 
	By choice of $\rt{a'b'|c'} \in R_2$ we have $a'\in \alpha$ and	$b'\in \beta$. 
	Since $A', \alpha$ and $\beta$ induce a connected subgraph in  
	$[R_1\setminus \{\rt{ab|c}\}, A\cup B\}]$, respectively, and since $a'\in A'\cap \alpha$ and 
	$b'\in  A'\cap \beta$, 
	the induced subgraph $\langle A'\cup \alpha \cup \beta \rangle$ is connected 
	in $[R_1\setminus \{\rt{ab|c}\}, A\cup B\}]$. However, since  $a\in \alpha$ 
	and $b\in \beta$, the triple $\rt{ab|c}$ is not a bridge in $[R_1,A\cup B]$; a contradiction to Lemma \ref{lem:bridge}.
	By analogous arguments one obtains a contradiction if $a',b'\in B'$.
	Therefore, there is a triple  $\rt{a'b'|c'} \in R_2$ such that $a'\in \alpha$, 	$b'\in \beta$ and	$c'\in \gamma$.
%    \hfill{$\diamond$}

	 \hfill $_{\textrm{\footnotesize{-- End Proof Claim 1 -- }}}$
  \end{description} 	
	In what follows, let $\rt{a'b'|c'} \in R_2$ be chosen such that $a'\in \alpha$, 	$b'\in \beta$ and	$c'\in \gamma$.
	\smallskip		 
	\begin{description}
	  \item[\emph{Claim 2:}] \emph{It holds that $\rt{a'b'|c'}\in R_2\setminus R_1$.}

    \emph{Proof of Claim 2.} 
		Recall that $\rt{ab|c}\in R_1\setminus R_2$ and thus, 
		the triples $\rt{ab|c}$ and $\rt{a'b'|c'}$ must be distinct.
		Assume for contradiction
		that $\rt{a'b'|c'}\in R_1$. In this case, one can easily verify that there
		are either two edges $(a,b)$ and $(a',b')$ in $[R_1,A\cup B]$ connecting 
		$\alpha$ and $\beta$ or, if $(a,b) = (a',b')$, then the edge $(a,b)$ is supported by two triples.
		In either case, 
		$\rt{ab|c}$ is not a bridge in $[R_1,A\cup B]$; a contradiction to Lemma \ref{lem:bridge}.
  %  \hfill{$\diamond$}
	\hfill $_{\textrm{\footnotesize{-- End Proof Claim 2 -- }}}$
  \end{description} 	
	In what follows, we set $\Rn \coloneqq (R_1\setminus \{\rt{ab|c}\}) \cup \{\rt{a'b'|c'}\}$.
	\smallskip		 
	\begin{description}
	  \item[\emph{Claim 3:}] \emph{
		It holds that $\Rn\in \SC(R)$.}

    \emph{Proof of Claim 3.} Clearly, $\Rn\subseteq R$. Hence,
		in order to show that $\Rn\in \SC(R)$ it remains to show that 
		$\cl(\Rn) = \cl(R)$. 
		To this end, recap that  $[R_1\setminus\{\rt{ab|c}\},A\cup B]$ has the connected 
		components  $\alpha,\beta,\gamma$	with $a,a'\in \alpha$, $b,b'\in \beta$ and $c,c'\in \gamma$.
		Moreover,  Lemma \ref{lem:subgraph} implies that $[R_1\setminus\{\rt{ab|c}\},A\cup B]$ is a subgraph of $[\Rn,A\cup B]$
		and thus, $\langle \alpha \rangle$ and $\langle \beta \rangle$ 
		remain connected subgraphs in $[\Rn,A\cup B]$. However, since $a',b',c'\in A\cup B$ and $\rt{a'b'|c'}\in \Rn$
		we have an additional edge in $[\Rn,A\cup B]$
		that connects $\langle \alpha \rangle$ and $\langle \beta \rangle$ 
		by the edge $(a',b')$. Hence, $A=\alpha \cup \beta$
		induces a connected subgraph in $[\Rn,A\cup B]$, 
		while $\langle B \rangle = \langle\gamma\rangle$ remains unchanged and thus still
		provides a connected component in $[\Rn,A\cup B]$. 
		In summary, $[\Rn,A\cup B]$ has two connected components, 
		where $a,b\in A$ and $c\in B$. Theorem \ref{thm:2cc} implies that
		$\rt{ab|c} \in \cl(\Rn)$. Application of 
		Lemma \ref{lem:identCl} yields
		 $\cl(\Rn) = \cl(R_1\setminus \{\rt{ab|c}\}) \cup \{\rt{a'b'|c'}\}) = \cl(R_1\cup \{\rt{a'b'|c'}\})$. 
		 Moreover, it holds that $\cl(R) = \cl(R_1) \subseteq  \cl(R_1\cup \{\rt{a'b'|c'}\}) = \cl(\Rn) \subseteq \cl(R)$
		and therefore, $\cl(R) = \cl(\Rn)$. Thus, $\Rn\in \SC(R)$.
	 %   \hfill{$\diamond$}
	\hfill $_{\textrm{\footnotesize{-- End Proof Claim 3 -- }}}$
  \end{description} 	
 	In what follows, we want to show that all triples $\rt{xy|z}\in \Rn$ are bridges in 
	$[\Rn,A''\cup B'']$  where $\lmax{\rt{xy|z}}{\Rn}=\{A'',B''\}$ (see Claim 5). 
	In this case, Lemma \ref{lem:bridge} would imply that $\Rn\in \minimal(\SC(R))$. 
	To this end, however, we first need to prove Claim 4. 

	\smallskip		 
	\begin{description}
	  \item[\emph{Claim 4:}] \emph{
		Assume there is a triple $\rt{xy|z}\in \Rn$ which is not a bridge in  $[\Rn,A''\cup B'']$  where $\lmax{\rt{xy|z}}{\Rn}=\{A'',B''\}$.
		Then, $\rt{xy|z}\neq \rt{a'b'|c'}$; $a',b',c'\in A''\cup B''$; $x$ and $y$ are connected by a path in in $[\Rn \setminus \{\rt{xy|z}\},A''\cup B'']$ and 
		every path $P_{xy}$ in $[\Rn \setminus \{\rt{xy|z}\},A''\cup B'']$
		contains the edge $(a',b')$; and $\{A'',B''\}\neq \{A,B\}$. }

   \emph{Proof of Claim 4.} 			
			Assume that  $\rt{xy|z}\in \Rn$ is not a bridge  in  $[\Rn,A''\cup B'']$. 
			First, we show that $\rt{xy|z}\neq \rt{a'b'|c'}$. 
			Assume for contradiction that $\rt{xy|z} = \rt{a'b'|c'}$. 
			Now, we show that in this case $\{A'',B''\} = \{A,B\}$. 
			Note, since $\lmax{\rt{a'b'|c'}}{\Rn}=\{A'',B''\}$ and $\Rn\in \SC(R)$, 
			we can apply  Theorem \ref{thm:repres} and conclude that $\{A'',B''\}= \lmax{\rt{a'b'|c'}}{R}$. 
			Since $\lmax{\rt{ab|c}}{R}=\{A,B\}$, we have 
			$\{A,B\},\{A'',B''\} \in \Lmax{R}{R}$.
		Moreover, since by construction  $a,b,a',b'\in A$ and $c'\in B$, 
			we have $A\cap (A''\cup B'')\neq \emptyset$
			and $B\cap (A''\cup B'')\neq \emptyset$. 
			Hence, we can argue analogously as 
			in the last part of the proof of Theorem \ref{lem:nonempty}
			to conclude that $\{A,B\} = \{A'',B''\}$.
			Now, since 	$\rt{xy|z} = \rt{a'b'|c'}$, the triple $\rt{a'b'|c'}$
			is not a bridge in $[\Rn,A''\cup B'']$. Thus, there is a path 
			$P_{a'b'}$	in $[\Rn \setminus \{\rt{a'b'|c'}\},A''\cup B'']  = 
			[\Rn \setminus \{\rt{a'b'|c'}\},A\cup B] = [R_1 \setminus \{\rt{ab|c}\},A\cup B]$. 
			However, this implies that $P_{a'b'}$ connects $a'\in \alpha$ and $b'\in\beta$ in
			$[R_1 \setminus \{\rt{ab|c}\},A\cup B]$ and therefore, $\rt{ab|c}$ is not 
			a bridge in $[R_1,A\cup B]$; a contradiction to $R_1\in \minimal(\SC(R))$
			and Lemma \ref{lem:bridge}.  
			Hence $\rt{xy|z}\neq \rt{a'b'|c'}$. 

			We continue to show that every path $P_{xy}$ in $[\Rn \setminus \{\rt{xy|z}\},A''\cup B'']$
			contains the edge $(a',b')$. 
 			Since $\rt{xy|z}\in \Rn$ is not a bridge  in  $[\Rn,A''\cup B'']$
			there must be a path $P_{xy}$ in 
			$[\Rn \setminus \{\rt{xy|z}\},A''\cup B'']$. Assume for contradiction that $P_{xy}$ does not contain the edge $(a',b')$.
			Hence, $P_{xy}$ still connects $x$ and $y$ in $[\Rn \setminus \{\rt{xy|z},\rt{a'b'|c'}\},A''\cup B'']$.
			Since $\Rn \setminus \{\rt{xy|z},\rt{a'b'|c'}\}\subseteq R_1 \setminus \{\rt{xy|z}\}$ and by Lemma \ref{lem:subgraph}, 
			the graph	$[\Rn \setminus \{\rt{xy|z},\rt{a'b'|c'}\},A''\cup B'']$ is a subgraph of 
			$[R_1 \setminus \{\rt{xy|z}\},A''\cup B'']$. Therefore,  the path 
			$P_{xy}$ connects $x$ and $y$ in $[R_1 \setminus \{\rt{xy|z}\},A''\cup B'']$. 
			Note, Theorem \ref{thm:repres} implies that $\lmax{\rt{xy|z}}{R_1} = \{A'',B''\}$.
			Since $\rt{a'b'|c'} \neq \rt{xy|z} \in \Rn$, we have $\rt{xy|z}\in R_1$. 
			But then $\rt{xy|z}$ is not a bridge in $[R_1,A''\cup B'']$; 
			a contradiction to Lemma \ref{lem:bridge}. 

			The latter, in particular, implies that  $a',b'\in A''\cup B''$. 
			Now, assume for contradiction that $c'\notin  A''\cup B''$. Hence,
			$[\Rn \setminus \{\rt{xy|z},\rt{a'b'|c'}\},A''\cup B''] = [\Rn \setminus \{\rt{xy|z}\},A''\cup B'']$.
			By the preceding arguments, every path $P_{xy}$ in $[\Rn \setminus \{\rt{xy|z}\},A''\cup B'']$
			contains the edge $(a',b')$.	Again, since $[\Rn \setminus \{\rt{xy|z},\rt{a'b'|c'}\},A''\cup B'']$ is a subgraph of 
			$[R_1 \setminus \{\rt{xy|z}\},A''\cup B'']$ this path is also contained in $[R_1 \setminus \{\rt{xy|z}\},A''\cup B'']$ and
			the triple $\rt{xy|z}$ is not a bridge in $[R_1,A''\cup B'']$; 
			a contradiction to Lemma \ref{lem:bridge} and $\lmax{\rt{xy|z}}{R_1}=\{A'',B''\}$. 
			Thus, $c'\in  A''\cup B''$	    

			Finally, we show that  $\{A'',B''\}\neq \{A,B\}$. 
			Assume for contradiction that $\{A'',B''\} = \{A,B\}$. W.l.o.g.\ let $A''=A = \alpha\cup \beta$ and $B''=B = \gamma$. 
			First, we show that neither $x\in \gamma$ nor $y\in \gamma$.
		   Assume w.l.o.g.\ that $x\in \gamma$. 
			Note the path  $P_{xy}$ with edge $(a',b')$ in $[\Rn \setminus \{\rt{xy|z}\},A''\cup B'']$ 
			is also contained $[\Rn,A''\cup B'']$. However, since $a',b'\in A''$ and $x\in \gamma=B''$
			this path $P_{xy}$ connects the two connected components $A'', B''$ in $[\Rn,A''\cup B'']$;
			a contradiction. \\
			We continue to show that neither $x,y\in \alpha$ nor $x,y\in \beta$.	
			Assume for contradiction that $x,y\in \alpha$. 
			Since $[\Rn \setminus \{\rt{xy|z}\},A''\cup B'']$ contains a path  $P_{xy}$ with edge $(a',b')$, 
			and $x,y\in \alpha$ there must be a second edge $(a'',b'')$ distinct from $(a',b')$ in    $P_{xy}$ where 
			$a''\in \alpha, b''\in \beta$. 	
			Since $\Rn \setminus \{\rt{xy|z}\} = (R_1\setminus\{\rt{ab|c},\rt{xy|z}\})\cup \{\rt{a'b'|c'}\}$, 
			$\{A'',B''\}=\{A,B\}$ and removal of  $\{\rt{a'b'|c'}\}$ would still preserve the edge $(a'',b'')$,
  			this edge $(a'',b'')$ must also be contained in $[R_1\setminus\{\rt{ab|c},\rt{xy|z}\}, A\cup B]$. 
			Since $[R_1\setminus\{\rt{ab|c},\rt{xy|z}\}, A\cup B]$ is a subgraph of 
			$[R_1\setminus\{\rt{ab|c}\}, A\cup B]$, the latter graph contains the edge $(a'',b'')$
			that connects the components $\alpha$ and $\beta$. 
			But then $\rt{ab|c}$ is not a bridge in $[R_1, A\cup B]$;
			   a contradiction to $R_1\in \minimal(\SC(R))$ and Lemma \ref{lem:bridge}. 
		Hence, $x$ and $y$ cannot be both in $\alpha$, and by similar arguments, not both in $\beta$. \\
		Thus, there are only two cases left: $x\in \alpha$ and $y\in \beta$, or 
		$y\in \alpha$ and $x\in \beta$. Assume w.l.o.g.\ that $x\in \alpha$ and $y\in \beta$.
		Since $\rt{xy|z}\in R_1\setminus \{\rt{ab|c}\}$, there must be the edge 
		$(x,y)$ in $[R_1\setminus  \{\rt{ab|c}\}, A\cup B]$, in which case 
		$\alpha$ and $\beta$ form a connected component. 
		Again, $\rt{ab|c}$ is not a bridge in $[R_1, A\cup B]$ and we obtain
	   a contradiction to $R_1\in \minimal(\SC(R))$ and Lemma \ref{lem:bridge}. \\
		Therefore, if $\{A'',B''\}=\{A,B\}$, then $x,y\notin \alpha\cup\beta\cup \gamma = A''\cup B''$; a contradiction since we 
		assumed that $\lmax{\rt{xy|z}}{\Rn}=\{A'',B''\}$ and hence, $x,y\in A''\cup B''$.
	\hfill $_{\textrm{\footnotesize{-- End Proof Claim 4 -- }}}$

	 \item[\emph{Claim 5:}] \emph{$\Rn \in \minimal(\SC(R))$.}

   \emph{Proof of Claim 5.} 	In order to show that $\Rn \in \minimal(\SC(R))$ we use Lemma \ref{lem:bridge}
		and show that each triple $\rt{xy|z}\in \Rn$ must be a bridge in  $[\Rn,A''\cup B'']$  
		where $\lmax{\rt{xy|z}}{\Rn}=\{A'',B''\}$.

		Assume for contradiction, that there is a triple $\rt{xy|z}\in \Rn$ that is not a bridge in 
  	   $[\Rn,A''\cup B'']$. Claim 4.\ implies that $\rt{xy|z}\neq \rt{a'b'|c'}$ and thus, in particular, 
		$\rt{xy|z}\in R_1\setminus \{\rt{ab|c}\}$. Moreover, $a',b',c'\in A''\cup B''$, $\{A'',B''\}\neq\{A,B\}$ and every path 
		$P_{xy}$ in $[\Rn \setminus \{\rt{xy|z}\},A''\cup B'']$
		contains the edge $(a',b')$. 
		Recap that $a,a'\in \alpha$, $b,b'\in \beta$, $c,c'\in \gamma$, $A=\alpha\cup \beta$
		and $B=\gamma$.
		
		Recap that $\lmax{\rt{ab|c}}{R_1} = \{A,B\}\in \Lmax{R_1}{R_1}$. 
		Claim 3 implies that $\Rn\in \SC(R)$. Thus, we can apply	
		Theorem \ref{thm:repres} and conclude that 
		$\{A'',B''\} =\lmax{\rt{xy|z}}{\Rn} = \lmax{\rt{xy|z}}{R} = \lmax{\rt{xy|z}}{R_1}$.
		Hence, $\{A'',B''\}\in \Lmax{R_1}{R_1}$. 
		Moreover, since $a',b',c'\in A''\cup B''$ as well as $a',b'\in A$ and $c'\in B$
		it holds that $A\cap(A''\cup B'')\neq \emptyset$ and $B\cap(A''\cup B'')\neq \emptyset$. 
		Thus, 	we can apply Lemma \ref{lem:intersection-subset} and conclude that 
		$A\cup B\subseteq A''$ or  $A\cup B\subseteq B''$. 
		W.l.o.g.\ assume $A\cup B\subseteq A''$.

		Denote one of the paths in $[\Rn \setminus \{\rt{xy|z}\},A''\cup B'']$ 
		that connect $x$ and $y$ by $P_{xy}$.
 		Claim 4 implies that $P_{xy}$ contains the edge $(a',b')$.
		Since $a',b'\in A''$ it must hold that  
		$x,y\in A''$ as otherwise  
		$P_{xy}$ would connect $A''$ and $B''$ in $[\Rn,A''\cup B'']$. 
		Therefore, $z\in B''$ and hence $z\notin A\cup B$. 
		Since  $P_{xy}$ contains the edge $(a',b')$, it can be decomposed into the paths 
		 $P_{xa'}$ and  $P_{b'y}$ (resp. $P_{xb'}$ and  $P_{ya'}$) and the edge  $(a',b')$. 
		W.l.o.g.\ assume that  $P_{xy}$ is composed of $P_{xa'}$, $(a',b')$ and  $P_{b'y}$.
 	Note, since neither $P_{xa'}$ nor $P_{b'y}$ contains the edge $(a',b')$, we can conclude that
		both paths are contained in 
		$[\Rn\setminus\{\rt{xy|z},\rt{a'b'|c'}\}, A''\cup B''] = 
		[R_1\setminus\{\rt{xy|z},\rt{ab|c}\}, A''\cup B''] \subseteq [R_1\setminus\{\rt{xy|z}\}, A''\cup B'']$.
		Furthermore, since  $\alpha$ and $\beta$
		induce connected subgraphs in $[R_1\setminus\rt{ab|c},A\cup B]$ and  $a,a'\in \alpha$, $b,b'\in \beta$,  
		there are paths $P_{aa'}$ and $P_{bb'}$ in  
		$[R_1\setminus\{\rt{ab|c}\},A\cup B]$. 
		Since $z\notin A\cup B$ and $A\cup B\subseteq A''$, we have 
		$[R_1\setminus\{\rt{ab|c}\},A\cup B] = [R_1\setminus\{\rt{ab|c}, \rt{xy|z}\},A\cup B] \subseteq[R_1\setminus\{\rt{xy|z}\},A''\cup B'']$. 
		Hence, the paths $P_{aa'}$ and $P_{bb'}$ are also contained in  $[R_1\setminus\{\rt{xy|z}\},A''\cup B'']$.
		In summary, $[R_1\setminus\{\rt{xy|z}\},A''\cup B'']$ contains the paths  $P_{aa'}$, $P_{bb'}$, $P_{xa'}$ and $P_{b'y}$
		but also the edge $(a,b)$, since $\rt{ab|c}\in R_1\setminus \{\rt{xy|z}\}$ and $a,b,c\in A\cup B\subseteq A''$. 
		Hence, we can combine the four paths and the edge $(a,b)$ to a walk 
		in $[R_1\setminus\{\rt{xy|z}\},A''\cup B'']$ that connects $x$ and $y$. 
		However, this implies that $\rt{xy|z}$ is not a bridge in $[R_1,A''\cup B'']$; 
		a contradiction to $\lmax{\rt{xy|z}}{R_1} =\{A'',B''\}$ and Lemma  \ref{lem:bridge}.

		In summary, for all cases for which there is  a triple $\rt{xy|z}\in \Rn$ that is not a bridge in 
  	   $[\Rn,A''\cup B'']$ we obtain a contradiction. Hence, 
		each triple $\rt{xy|z}\in \Rn$ must be a bridge in   $[\Rn,A''\cup B'']$
		and we can apply Lemma \ref{lem:bridge} to conclude that  $\Rn \in \minimal(\SC(R))$.  
	%\hfill{$\diamond$}
	\hfill $_{\textrm{\footnotesize{-- End Proof Claim 5 -- }}}$
 \end{description} 	

We have shown that for any $R_1,R_2\in \mc{B} = \minimal(\SC(R))$ and 	$\rt{ab|c}\in R_1\setminus R_2$
there is a triple   $\rt{a'b'|c'}\in R_2\setminus R_1$ such that 
$\Rn = (R_1\setminus \{\rt{ab|c}\}) \cup \{\rt{a'b'|c'}\} \in \mc{B}$. 
Hence, we can apply Lemma \ref{lem:basis} to conclude that $(R,\mathbb{F}_R)$ is a matroid.
\end{proof}

In order to avoid confusion, we emphasize that  
the closure operator $\cl(R)$ for rooted triple sets $R$ 
defined here is not a matroid closure operator $\cl_M$ \cite{Oxley:11,korte2012combinatorial}. 
Note, since $M = (R,\mathbb F_R)$ is a matroid, the following property must be satisfied for $\cl_M$, 
cf.\ \cite[Lemma 1.4.3]{Oxley:11}:
\[X\subseteq R, r\in R \text{ and } r'\in \cl_M(X\cup\{r\})\setminus \cl_M(X) \implies r\in \cl_M(X\cup\{r'\}).\,\]
To see that $\cl(R)$ does not fulfill this property in general, 
consider the example in Figure \ref{fig:exmpl}. To recap, 
$R_1 = \{\rt{ab|c},\rt{ac|d}\}$, $R_3=\{\rt{ac|d},\rt{bc|d}\}$
and $\cl(R_1)=\cl(R)=\{\rt{ab|c},\rt{ac|d},\rt{bc|d},\rt{ab|d}\}$,
but $\cl(R_3) = \cl(R)\setminus \rt{ab|c}$. 
Now, put $X=\{\rt{ac|d}\}$ and $r=\rt{ab|c}$. Thus, 
$r' = \rt{bc|d}\in \cl(X\cup\{r\})\setminus \cl(X) = \cl(R_1)\setminus \{\rt{ab|c}\}$. 
However, $r = \rt{ab|c}\notin \cl(X\cup \{r'\}) =\cl(R_3)$.
The latter result has already been observed by David Bryant \cite{Bryant97}, however, 
the matroid structure of $(R,\mathbb F_R)$ was not discovered.

Note, each minimum representative set $R'\in \minimum(\SC(R))$ is also minimal. 
Thus, $\minimum(\SC(R))\subseteq \minimal(\SC(R))$. However, since 
$(R,\mathbb{F}_R)$ is a matroid with collection of bases $\minimal(\SC(R))$, all elements in 
$\minimal(\SC(R))$ have the same cardinality \cite{Oxley:11}. Therefore,
all basis elements of the matroid $(R,\mathbb{F}_R)$ are of minimum size. 
We summarize this observation in
the following

\begin{theorem}
If $R$ is a consistent triple set, then $\minimal(\SC(R)) = \minimum(\SC(R))$.
\label{thm:min=MIN}
\end{theorem}

In order to find a minimum representative set $R'$ of $R$ one can apply a simple greedy algorithm.  
Algorithm \ref{alg:greedy} computes a basis element of the matroid  $(R,\mathbb{F}_R)$
		and can easily be adapted to find maximum weighted bases, an issue that might 
		be important for applications in phylogenetics, where the weight of a rooted triple corresponds
		to a statistical confidence value or any other measure associated with the underlying triples. 

\begin{lemma}
	Algorithm \ref{alg:greedy} computes a subset $R'\subseteq R$ with
	$\cl(R') = \cl(R)$
	of minimum size in $\mc{O}(|R|^2|L_R|)$.  
	\label{lem:greedy}
\end{lemma}
\begin{proof}
	By Lemma \ref{lem:algo-cl}, 
	it suffices to decide whether a triple 
	$\rt{ab|c}$ is contained in $\cl(R\setminus (R_{\tmp}\cup\{\rt{ab|c}\}))$
	by the two consistency checks in the IF-condition. 

	Let $R_{\tmp} = \{r_1,\dots,r_k\}$ where the indices of the triples	are chosen w.r.t.\ the order in which they are added to $R_{\tmp}$.
	By construction, $r_i\in R_{\tmp}$ if $r_i\in \cl(R\setminus \{r_1,\dots,r_i\})$.
	Lemma \ref{lem:identCl} implies that for 
	the first triple  $r_1\in \cl(R\setminus \{r_1\})$ it holds that 
	$\cl(R\setminus \{r_1\}) = \cl(R)$. Next, $r_2$ is added to $R_{\tmp}$ 
	that is, $r_2\in \cl(R\setminus \{r_1,r_2\})$ and again by Lemma \ref{lem:identCl},
	 $\cl(R\setminus \{r_1,r_2\}) = \cl(R\setminus \{r_1\}) = \cl(R)$.
	Inductively, when $r_k$ is
	chosen we have 
	$r_k\in \cl(R\setminus R_{\tmp}\}) = 
	\cl(R\setminus \{r_1,\dots,r_{k-1}\}) = \cdots = \cl(R\setminus \{r_1\}) = \cl(R)$. 
	Since by construction, $R_{\min} = R\setminus R_{\tmp}$, it holds that  
	$\cl(R_{\min}) = \cl(R\setminus R_{\tmp}) = \cl(R)$. 

	We continue to show that $R_{\min} \in \minimal(\SC(R))$.
	Assume for contradiction that 
	there is a subset $R''\subsetneq R_{\min}$	with $\cl(R'')   = \cl(R)$.
	Note, $R'' =  R_{\min} \setminus R'$ for some non-empty subset $R'\subseteq R_{\min}$. 
	Thus, $\cl( R_{\min} \setminus R')   = \cl(R) =  \cl(R_{\min})$.
	Lemma \ref{lem:identCl} implies that there is a triple $r\in R'$ such
	that  $\cl( R_{\min} \setminus \{r\}) =  \cl( R_{\min})$. 
	Note, since $r\in R'\subseteq R_{\min} =R\setminus R_{\tmp}$ it holds that $r\notin  R_{\tmp}$. 

	Consider the step when $r$ is chosen 	in  Alg.\ \ref{alg:greedy}. 
	If $R_{\tmp}=\emptyset$ before this step, 
	we would have $r\notin \cl(R\setminus \{r\})$, since $r$ is not added to 
	$R_{\tmp}$.
	However, since $r\in R_{\min}$ and $R_{\min} \setminus \{r\} \subseteq R \setminus \{r\}$ and  it must hold that
	$r\in \cl(R_{\min}) =\cl( R_{\min} \setminus \{r\}) \subseteq  \cl(R \setminus \{r\})$;
	a contradiction. 
	If $R_{\tmp}$ is not empty and thus, $R_{\tmp}=\{r_1,\dots,r_i\}$ 
	before the step 
	when $r$ is chosen in  Alg.\ \ref{alg:greedy}, we would have 
	$r\notin \cl(R\setminus \{r_1,\dots,r_i,r\})$, since $r$ is not added to 
	$R_{\tmp}$. 
	However, since $r\in R_{\min}$ and $R_{\min} \setminus \{r\} = R\setminus \{r_1,\dots,r_k,r\} \subseteq R\setminus \{r_1,\dots,r_i,r\}$
	it must hold that $r\in \cl(R_{\min}) = \cl( R_{\min} \setminus \{r\})	\subseteq \cl(R\setminus \{r_1,\dots,r_i,r\})$;
	a contradiction. 
	Therefore, $R_{\min}$ is minimal and we can apply Theorem \ref{thm:min=MIN}
	to conclude that  $R_{\min}$ is of minimum size.

	Concerning the time complexity, observe that the for-loop runs $|R|$ times. 
	In each step of the for-loop, we have to check for consistency which can be done with 
	\texttt{BUILD} in $\mc{O}(|R||L_R|)$ time. Thus, we end in an overall
	time complexity $\mc{O}(|R|^2|L_R|)$.
\end{proof}

\begin{algorithm}[tbp]
\caption{\texttt{GREEDY for Minimal/Minimum Representative Triple Sets}}
\label{alg:greedy}
\begin{algorithmic}[1]
  	\Require Consistent triple set $R$; 
    \Ensure  Minimal Representative Triple set $R_{\min}$;  
	\State 	$R_{\tmp}\gets \emptyset$;
	\For{all $\rt{ab|c}\in R$}
		\State $R'\gets R\setminus R_{\tmp}$;  \label{alg:1}
    	\If{$R'\setminus \{\rt{ab|c}\} \cup \{\rt{bc|a}\}$ and $R'\setminus \{\rt{ab|c}\} \cup \{\rt{ac|b} \}$ are not consistent}   
		 \Comment{Thus, $\rt{ab|c}\in \cl(R\setminus (R_{\tmp}\cup\{\rt{ab|c}\}))$}    
			\State $R_{\tmp} \gets R_{\tmp} \cup \{\rt{ab|c}\}$;
		\EndIf
	\EndFor
	\State \Return $R_{\min}\gets R\setminus R_{\tmp}$;
\end{algorithmic}
\end{algorithm}

As a consequence, we obtain the following result:
\begin{theorem}
Let $R_1, R_2$ be consistent triple sets such that $\cl(R_1)=\cl(R_2)$.  
For each $R'_1\in \minimal(\SC(R_1))$ and  $R'_2 \in \minimal(\SC(R_2))$
it holds that $|R'_1| = |R'_2|$.
\label{thm:min=MIN2}
\end{theorem}
\begin{proof}
	Let $R_1$ and $R_2$ be consistent triple sets such that $\cl(R_1)=\cl(R_2)$. 
	Set $R  = \cl(R_1)$ and apply the greedy method with input $R$. 
	Since $\cl(R_1) = \cl(R)$ and since the choice of the triples assigned to 
	$R_{\tmp}$ is arbitrary as long as $\cl(R\setminus R_{\tmp}) = \cl(R)$, it is possible
	to obtain greedily a set  $R_{\tmp}$ for which
	$R'=R_1= R\setminus R_{\tmp}\subseteq R$ (in Step \ref{alg:1} of Alg.\ \ref{alg:greedy}). 
	Now Alg.\ \ref{alg:greedy} continues with $R_1$ in order to find a subset $R'\subseteq R_1$
	such that $R'\in \minimal(\SC(R)) =\minimal(\SC(\cl(R_1)))$. 
	Note, $R'\in \minimal(\SC(R_1))$ as otherwise there would be a subset
	$R''\subsetneq R'\subseteq R_1\subseteq \cl(R_1)$ such that
	$\cl(R'') = \cl(R_1)$; a contradiction to $R'\in \minimal(\SC(\cl(R_1)))$ and the correctness of Alg.\ \ref{alg:greedy}.
	Hence, 	for any $R'_1\in \minimal(\SC(\cl(R_1)))$ with $R'_1\subseteq R_1$ we also have
	$R'_1\in \minimal(\SC(R_1))$. The same applies to $R_2$, that is, $R'_2\in \minimal(\SC(R_2))$
	for any $R'_2\in \minimal(\SC(\cl(R_2)))$ with $R'_2\subseteq R_2$. 
	Since $\cl(R_1)=\cl(R_2)$, it holds that $\minimal(\SC(\cl(R_1))) = \minimal(\SC(\cl(R_2)))$.
	The latter together with Theorem \ref{thm:min=MIN} implies that $|R'_1| = |R'_2|$. 
\end{proof}

\section{Computing the Closure}
\label{sec:cl}

The currently fastest algorithm to determine the closure has a time complexity of
$\mc{O}(|R||L_R|^4)$ and was proposed by Bryant and Steel \cite{BS:95}. 
In this section, we provide a novel and efficient algorithm to compute the closure.
This method  is based on the techniques we used to prove the matroid structure. 
In particular, the proposed algorithm will rely on computing the set $\Lmax{R}{R}$ and usage of the following theorem.

\begin{theorem}
	Let $R$ be a consistent triple set and define $\R_{A,B} = \{\rt{ab|c} \colon a,b\in A, c\in B \text{ or } a,b\in B, c\in A \}$
	for any $A,B\subseteq L_R$. Then, 
	\[\cl(R) = \bigcup_{\{A,B\}\in \Lmax{R}{R}} \R_{A,B}.\,\]
	Moreover, for any distinct $\{A,B\},\{A',B'\} \in \Lmax{R}{R}$
	it holds that 	$\R_{A,B} \cap \R_{A',B'} = \emptyset$.
	In particular, \[\sum_{\{A,B\}\in \Lmax{R}{R}} |R_{A,B}| \leq |L_R|^3.\,\]
	\label{thm:cl-R}
\end{theorem}
\begin{proof}
	Theorem \ref{thm:2cc} immediately implies that 	
	$ \bigcup_{\{A,B\}\in \Lmax{R}{R}} \R_{A,B} \subseteq \cl(R)$. 
	Thus, it remains to show that 
		$ \cl(R) \subseteq \bigcup_{\{A,B\}\in \Lmax{R}{R}} \R_{A,B} $. 
	Let $\rt{ab|c}\in\cl(R)$. Lemma \ref{lem:nonempty} implies that 
	$\LrR{\rt{ab|c}}{R}\neq \emptyset$.
	Thus, there is also a maximal element $\lmax{\rt{ab|c}}{R}=\{A,B\} \in \Lmax{\cl(R)}{R}$. 
	 Theorem  \ref{thm:equals} implies that $\Lmax{\cl(R)}{R} = \Lmax{R}{R}$ and hence, 
	it particularly holds that $\{A,B\} \in \Lmax{R}{R}$ for which 
	$\rt{ab|c}\in \R_{A,B}$.
	 Therefore, 
	$\rt{ab|c}\in 
	  \bigcup_{\{A,B\}\in \Lmax{R}{R}} \R_{A,B}$ and thus, 
		$\cl(R) = \bigcup_{\{A,B\}\in \Lmax{R}{R}} \R_{A,B}$. 

	We continue by showing that for any distinct $\{A,B\},\{A',B'\} \in \Lmax{R}{R}$
	we have 	$\R_{A,B} \cap \R_{A',B'} = \emptyset$.
	Assume for contradiction that $\rt{ab|c}\in \R_{A,B} \cap \R_{A',B'}$
	for some distinct $\{A,B\},\{A',B'\} \in \Lmax{R}{R}$.
	Thus, $A\cap (A'\cup B')\neq \emptyset$ and $B\cap (A'\cup B')\neq \emptyset$, 
	as well as $A'\cap (A\cup B)\neq \emptyset$ and $B'\cap (A\cup B)\neq \emptyset$. 
	Lemma \ref{lem:intersection-subset} implies 
	that 	either	$A\cup B\subseteq A'$ or  $A\cup B\subseteq B'$
	and either $A'\cup B'\subseteq A$ or  $A'\cup B'\subseteq B$.
	W.l.o.g.\ assume  that $A'\cup B'\subseteq A$ and therefore, 
	$A', B'\subseteq A$. 
	If $A\cup B\subseteq A'$ (resp.\ $A\cup B\subseteq B'$), 
	then $B'\subseteq A'$ (resp.\ $A'\subseteq B'$); 
	a contradiction to the disjointedness of $A',B'$. 
	Hence, $\R_{A,B} \cap \R_{A',B'} = \emptyset$.
	
	Finally, since $\R_{A,B} \cap \R_{A',B'} = \emptyset$
	are disjoint for distinct $\{A,B\},\{A',B'\} \in \Lmax{R}{R}$, 
	the closure $ \cl(R)$ is the disjoint union 
	$\uplus_{\{A,B\}\in \Lmax{R}{R}} \R_{A,B} $ and therefore, 
	$|\cl(R)| =  \sum_{\{A,B\}\in \Lmax{R}{R}} |R_{A,B}|$. 
	Since $\cl(R)$ can have at most $|L_R|^3$ triples, that is, 
	one triple for each of three-element subsets of $L_R$, 
	the assertion follows. 
\end{proof}

\begin{lemma}
	Let $R$ be a consistent triple set and $\rt{ab|c}\in R$. 
	Moreover, assume that there is a subset $\LT\subseteq L_R$
	with $a,b,c\in \LT$ such that $a, b$ are contained together 
	in some connected component $\mc{C}_{a,b}$ of $[R,\LT]$. 
	Let $\mc{C}_c$ denote the connected component in $[R,\LT]$
	that contains $c$. Note, we don't claim that 
	$\mc{C}_{a,b}\neq \mc{C}_c$. 
	
	Then, $\LT'\subseteq \mc{C}_{a,b}\cup \mc{C}_c$ for all $\LT'\subseteq \LT$
	for which $[R,\LT']$ has exactly two connected components, one containing
	$a,b$ and the other $c$. 
	\label{lem:abc}
\end{lemma}
\begin{proof}
	Assume for contradiction that there is a subset $\LT'=A\cup B\subseteq \LT$
	such that  $[R,\LT']$ has exactly two connected components $A$ and $B$ with
	$a,b\in A$ and $c\in B$, but $\LT'\not\subseteq \mc{C}_{a,b}\cup \mc{C}_c$.
	Thus, there is a vertex $d\in \LT'\setminus (\mc{C}_{a,b}\cup \mc{C}_c)$. 
	Therefore, $d$ is either contained in $A$ or in $B$, that is, 
	there is either a path $P_{da}$ or $P_{dc}$ in 
	$[R,\LT']$. Since $[R,\LT']$ is a subgraph of $[R,\LT]$
	these paths are also contained in $[R,\LT]$. 
	But then, $d\in \mc{C}_{a,b}$ or $d\in\mc{C}_c$; a contradiction. 
\end{proof}

\begin{algorithm}[tbp]
\caption{\texttt{Compute $\Lmax{R}{R}$}}
\label{alg:lmax}
\begin{algorithmic}[1]
  	\Require A consistent triple set $R$; 
    \Ensure $\Lmax{R}{R}$;  
	\State $\Lmax{R}{R}\gets \emptyset$;  	
	\For{all $\rt{ab|c}\in R$}
		\State $\mc{C}	\gets L_{R}$;
		\While{$[R,\mc{C}]$ does not have exactly two connected components $A,B$, one containing $a,b$ and the other $c$}
		\State	$\mc{C}_{a,b}\gets$ connected component in 	$[R,\mc{C}]$ that contains $a,b$;
		\State	$\mc{C}_{c}\gets$ connected component in 	$[R,\mc{C}]$ that contains $c$;
		\State	$\mc{C}\gets \mc{C}_{a,b} \cup \mc{C}_c$;
		\EndWhile 
		\State $\Lmax{R}{R}\gets \Lmax{R}{R}\cup \{A,B\}$;
	\EndFor
	\State \Return $\Lmax{R}{R} \gets \Lmax{R}{R} $;
\end{algorithmic}
\end{algorithm}

\begin{algorithm}[tbp]
\caption{\texttt{Compute Closure  $\cl(R)$}}
\label{alg:closure}
\begin{algorithmic}[1]
  	\Require A consistent triple set $R$; 
    \Ensure $\cl(R)$;  
	\State Compute $\Lmax{R}{R}$ with Algorithm \ref{alg:lmax};
	\State $\cl(R)\gets \emptyset$;
	\For{all $\{A,B\}\in \Lmax{R}{R}$}
		\State Compute 	$\R_{A,B}$ (cf.\, Theorem \ref{thm:cl-R});
		\State $\cl(R) \gets \cl(R) \cup \R_{A,B}$; 
	\EndFor
	\State \Return $\cl(R)$;
\end{algorithmic}
\end{algorithm}

The latter lemma immediately offers a way  to compute 
$\Lmax{R}{R}$ that is summarized in Algorithm \ref{alg:lmax}. 
For each triple $\rt{ab|c}\in R$ start with $[R,L_R]$. 
If  $[R,L_R]$ has already  two connected components $A$ and $B$,
one containing $a,b$ and the other $c$, then $L_R = A\cup B$
clearly maximizes  $|A\cup B|$. Thus, $\{A, B\} = \lmax{\rt{ab|c}}{R}\in \Lmax{R}{R}$.
If  $[R,L_R]$ does not have these two connected components $A$ and $B$,
it is, however, still disconnected (cf.\, Theorem \ref{thm:ahograph}). 
Hence,  $[R,L_R]$ has two or more   connected components. Nevertheless, 
$a,b$ must be in one connected component $\mc{C}_{a,b}$ due to the edge $(a,b)$
implied by $\rt{ab|c}\in R$. Moreover, there is a connected component  $\mc{C}_{c}$
that contains $c$. Note, $\mc{C}_{a,b} = \mc{C}_c$ might be possible.
Now set $\mathcal C = \mc{C}_{a,b}\cup \mc{C}_c\subsetneq L_R$. 
By Lemma \ref{lem:abc} it holds $\LT'\subseteq \mc{C}$ for all $\LT'\subseteq L_R$
for which $[R,\LT']$ satisfies the conditions of Theorem \ref{thm:2cc} when applied
to $\rt{ab|c}\in R$. Hence, we stepwisely look at these components $\mc{C}$
until we have found one $\mc{C}$ such that for the particular subset 
$\mc{C}^* = \mc{C}_{a,b} \cup \mc{C}_c \subsetneq \mc{C}$, 
the Ahograph $[R,\mc{C}^*]$ has two connected components $A$ and $B$,
one containing $a,b$ and the other $c$. Hence, $\mc{C}^* = A\cup B$. 
Since this is the first occurrence of such a set $\mc{C}^*\subseteq L_R$
and any further $\LT'\subseteq \LT$
for which $[R,\LT']$ has exactly two connected components, one containing
$a,b$ and the other $c$, must be contained in $\mc{C}^*$, 
$\mc{C}^* = A\cup B$ maximizes  $|A\cup B|$. 
Thus, $\{A, B\} = \lmax{\rt{ab|c}}{R}\in \Lmax{R}{R}$.
By Theorem \ref{thm:2cc} and since  $\rt{ab|c}\in R\subseteq \cl(R)$, 
there is indeed a subset $[R,\LT']$ that satisfies the conditions of Theorem \ref{thm:2cc} when applied
to $\rt{ab|c}\in R$. 
The latter arguments show that Algorithm \ref{alg:lmax} is correct.

\begin{lemma}
	Let $R$ be a consistent triple set.
	Algorithm \ref{alg:lmax} computes $\Lmax{R}{R}$ 
	in $\mc{O}(|R|^2|L_R|)$ time. 
	\label{lem:alg:lmax}
\end{lemma}
\begin{proof}
	The correctness of the Algorithms follows from Lemma \ref{lem:abc} and the discussion above. 
	
	The FOR-loop runs $\mc{O}(|R|)$ times.  
	The WHILE-condition is repeated at most $|L_R|$ times, since $\mc{C}$
	will  in each step have at least one vertex less as otherwise 
	$[R,\mc{C}]$ will be connected; a contradiction to Theorem \ref{thm:ahograph}. 
	For each call of the WHILE-condition we have to construct
	$[R,\mc{C}]$ and keep track of the connected components 
	$\mc{C}_{a,b}$ and $\mc{C}_{c}$. The latter task can be done 
	while constructing $[R,\mc{C}]$. Thus, both tasks take
	$\mc{O}(|L_R|+|R|)$ time. Since $|L_R|\leq 3|R|$, we have 
	$\mc{O}(|L_R|+|R|) = \mc{O}(|R|)$. 
	Thus, we end in an overall time complexity 
	$\mc{O}(|R|^2|L_R|))$.
\end{proof}

\begin{lemma}
	Algorithm \ref{alg:closure} computes the closure $\cl(R)$ of a consistent triple set $R$
	in $\mc{O}(|R|^2|L_R|)$ time.  
\end{lemma}
\begin{proof}
	Given a consistent triple set $R$, the set $\Lmax{R}{R}$ is computed. 
	For each $\{A,B\}\in \Lmax{R}{R}$ the respective set $\R_{A,B}$ 
	is constructed and attached to $\cl(R)$. By Theorem  \ref{thm:cl-R}, 
	$\cl(R)$ is correctly computed. 
	
	Concerning the  time complexity, first observe  that Algorithm \ref{alg:lmax}
	runs in  $\mc{O}(|R|^2|L_R|)$ time.
	Furthermore, Theorem \ref{thm:cl-R} implies that 
	$\R_{A,B} \cap \R_{A',B'} = \emptyset$ for distinct $\{A,B\},\{A',B'\} \in \Lmax{R}{R}$. 
	That is, each triple of $\cl(R)$ is computed exactly once in the entire run of the FOR-loop.
	Since $\cl(R)$ can have at most $|L_R|^3$ triples, the FOR-loop has a time complexity 
	of $\mc{O}(|L_R|^3)$. 
	Thus, we end in an overall time complexity $\mc{O}(|R|^2|L_R| + |L_R|^3)$.
	Since $|L_R|\leq 3|R|$, we have therefore $\mc{O}(|R|^2|L_R| + |L_R|^3) = \mc{O}(|R|^2|L_R|)$.
\end{proof}

The overall time complexity to compute the closure for a given triple set
$R$ is $\mc{O}(|R|^2|L_R|)$. In a worst case, $|R|$ is close to
$\binom{L_R}{3} = \mc{O}(|L_R|^3)$, in which we end in $\mc{O}(|L_R|^7)$
time. In this case, the time complexity of our approach is the same as the
complexity $\mc{O}(|R||L_R|^4) = \mc{O}(|L_R|^7)$ of the method proposed by
Bryant and Steel \cite{BS:95}. In a best case, however, we have
$\mc{O}(|R|)=\mc{O}(|L_R|)$ and then we obtain $\mc{O}(|R|^2|L_R|)=
\mc{O}(|L_R|^3)$, while the method of Bryant and Steel has a time
complexity of $\mc{O}(|R||L_R|^4) = \mc{O}(|L_R|^5)$. Thus, although the
time complexities are asymptotically the same whenever $|R|$ is close to
$\binom{L_R}{3}$, our methods outperforms the approach of Bryant and Steel
\cite{BS:95} for moderately sized $R$. In particular, since $|L_R|\leq
3|R|$, our method has always time complexity $\mc{O}(|R|^2|L_R|) \subseteq
\mc{O}(|R|^3)$, while the method of Bryant and Steel has complexity
$\mc{O}(|R||L_R|^4) \subseteq \mc{O}(|R|^5)$. 

For the sake of time complexity one can apply Algorithm \ref{alg:lmax} and
\ref{alg:closure} directly on an arbitrary given set $R'\in
\minimum(\SC(R))$, since $\Lmax{R}{R} = \Lmax{R'}{R'}$ (cf.\ Thm.\
\ref{thm:repres}). Note, in many cases $R'\in \minimum(\SC(R))$ will have
cardinality strictly less than $|R|$. By way of example, consider the set
of all rooted triples $\mathcal{R}(T)$ that are displayed by a binary
rooted tree $T$. In this case, $\mathcal{R}(T)$ is closed and defines $T$
and for any $R' \in \min(\SC(\mathcal{R}(T)))$ we have $|R'| = |L_R|-2$
\cite{GSS:07,Steel1992}. Hence, for any subset $R\subseteq \mathcal{R}(T)$
that contains $R'$ the cardinality can vary between $|L_R|-2$ and
$\binom{L_R}{3} \in \mc{O}(|L_R|^3)$. Note, $\cl(\mathcal{R}(T)) =
\cl(R')\subseteq \cl(R) \subseteq \cl(\mathcal{R}(T)) = \mathcal{R}(T)$ and
thus, $\cl(R') = \cl(R) = \mathcal{R}(T)$ and $R' \in \min(\SC(R))$.
Therefore, for any such set $R\subseteq \mathcal{R}(T)$ there is a minimal
representative set $R'$ that can have cardinality significantly less
compared to $|R|$, while $R'$ still contains all information of the tree
structure $T$ that is also provided by $R$ and $\mathcal{R}(T)$. On the one
hand, this strongly reduces the space complexity to store the information
that is needed to recover $T$. On the other hand, the closure can be
computed in the latter case in $\mc{O}(|R'|^2|L_R|)=\mc{O}(|L_R|^3)$ time,
whenever $R'$ is given, which improves the time complexity
$\mc{O}(|R||L_R|^4)$ to compute $\cl(R)$ by a factor of $|R||L_R|$. A
similar argument applies to the set of all triples $\mathcal{R}(T)$ that
are displayed by a non-binary rooted tree $T$. In this case,
$\mathcal{R}(T)$ identifies $T$. Still, $\mathcal{R}(T)$ can be close to
$\binom{L_R}{3} \in \mc{O}(|L_R|^3)$, while for any $R' \in
\min(\SC(\mathcal{R}(T)))$ the cardinality is bounded by $B(T) \in
\mc{O}(|L_R|^2)$, cf.\ Cor.\ \ref{cor:idn}.

\section{Further Results}
\label{sec:fr}

\subsection{Sufficient Conditions for Minimum Representative Triple Sets}

We provide in this subsection easy verifiable conditions that 
are quite useful to identify triples in $R$ that must be contained in every
representative set of $R$ and to check whether 
$R$ is already a minimal representative of $R$. 
In particular, these conditions helped us 
to construct many (counter)examples when we wrote this paper.
For instance, the sets $R'_1$ and $R'_2$ in Figure \ref{fig:exmpl2}
are easily verified to be minimal representatives by using the following results.

\begin{lemma}
	Let $R$ be a consistent triple set and $\rt{ab|c}\in R$. 
	Furthermore, let 	$C_{a,b}$ (resp.\ $C_c$) be the connected component in $[R,L_R]$
	that contains $a$ and $b$ (resp.\ $c$). Note,  $C_{a,b} = C_c$ is possible. 
	Then, \[\rt{ab|c} \in \bigcap_{R'\in\minimal(\SC(R))} R'\] whenever the following two conditions are satisfied:
	\begin{enumerate}	
       \item[(S1)] $[R,L_R]$ does not contain cycles.
 			 \item[(S2)]  If $\rt{ab|d}\in R$ with $c\neq d$, then $d\notin C_{a,b}\cup C_c$.
	\end{enumerate}
	Moreover, it holds that
	\[\bigcap_{R'\in\minimal(\SC(R))} R'  = \bigcap_{R'\in \SC(R)} R'.\] 
\label{lem:suff1}
\end{lemma}
\begin{proof}
	Assume that Conditions S1 and S2 are satisfied, but that there exists a triple set
	$R'\in\minimal(\SC(R))$ such that $\rt{ab|c}\notin R'$.
 	Let $\{A,B\} = \lmax{\rt{ab|c}}{R'}$ and assume  w.l.o.g.\ that $a,b\in A$ and $c\in B$. 
	Lemma \ref{lem:abc} implies that $A\cup B\subseteq C_{a,b}\cup C_c$. 
	Lemma  \ref{lem:subgraph} implies that $[R',A\cup B]$ is a subgraph of $[R,L_R]$.

	Since $a,b\in A$ there is a path $P_{ab}$ from $a$ to $b$ in $[R',A\cup B]$.
	Note, this path must be the edge $(a,b)$, otherwise $[R,L_R]$ would contain a cycle. 
	Thus, there must be 
	another triple $\rt{ab|d} \in R'$ with $c\neq d$ and $d\in A\cup B\subseteq C_{a,b}\cup C_c$
	that supports the edge $(a,b)$ in $[R',A\cup B]$; a contradiction to Condition S2. 

	To verify the last equation, 
	we set $M = \bigcap_{R'\in\minimal(\SC(R))} R'$ and $N = \bigcap_{R'\in \SC(R)} R'$. 
	Observe first that $\minimal(\SC(R)) \subseteq  \SC(R)$ implies that 
	$N\subseteq M$. Now assume that there is a triple $r\in R$ such $r\notin N$. 
	Thus, there is a triple set $R'\in \SC(R)$ with $r\notin R'$. 
	Therefore, $r\notin R''$ for all $R''\in \minimal(\SC(R'))$. 
	Since $R''$ is already minimal and $\cl(R'') = \cl(R') =\cl(R)$, we have $r\notin M$. 
	Hence, $M\subseteq N$ and thus, $M=N$.
\end{proof}

\begin{cor}
	Let $R$ be a consistent triple set. 
	Then, $\minimal(\SC(R))=\{R\}$ whenever Condition S1 and S2  in Lemma \ref{lem:suff1}
	are satisfied for	all triples in $R$.
\end{cor}

Note, the example in Figure \ref{fig:Wexmpl} shows that the Conditions S1 and S2 are
not necessary for minimal representatives.

\subsection{Triple Sets that Define and Identify a Tree}

Here, we are concerned with results established in  \cite{Steel1992} and \cite{GSS:07}.
First recall,  a triple set $R$  identifies a rooted tree $T$ with leaf set $L_R$, if
$T$ displays $R$  and any other tree that displays $R$ refines $T$.

 In  \cite{GSS:07} a tight lower bound for
the cardinality of triple sets that identify a rooted tree was given. To this end, let $c(v)$ denote the number of children
of a vertex $v$ in $T=(V,E)$ and set
\[B(T) = \sum_{(u,v)\in E^0} (c(u)-1)(c(v)-1).\,\]

\begin{theorem}[\cite{GSS:07}]
	Let $R$ be a consistent triple set. The following properties are satisfied:
	\begin{enumerate}
		\item $\cl(R) = \mathcal{R}(T)$ if and only if $R$ identifies $T$.
		\item If $R$ identifies $T$, then $|R|\geq B(T)$.
		\item For every rooted tree $T$, there is a triple set $R$
				such that $R$ identifies $T$ and $|R| = B(T)$.
	\end{enumerate}
	\label{thm:boundary}
\end{theorem}

These results allow us to give an exact value and  an upper bound  for the cardinality of minimal representative
triple sets that identify a tree. 
\begin{theorem}
Let $T$ be a tree with maximum degree $\Delta$.
Any minimal (and thus, minimum) consistent triple set $R$ that identifies $T$ has cardinality $B(T)\in 
\mc{O}(\Delta\cdot|L_R|)\subseteq \mc{O}(|L_R|^2)$.
\label{thm:cor:boundary}
\end{theorem}
\begin{proof}
Let $R$ be a minimal consistent triple set $R$ that identifies $T$.
Theorem \ref{thm:boundary}(1) implies that $R\in \minimal(\SC(\mathcal{R}(T)))$.
Theorem \ref{thm:min=MIN} implies that $R$ has minimum cardinality. 
Combining Theorem \ref{thm:boundary}(2,3) and Theorem \ref{thm:min=MIN} implies that each $R\in \minimal(\SC(\mathcal{R}(T)))$
has cardinality $|R| = B(T)$. 

We continue to show that $|B(T)|\in \mc{O}(|L_R|^2)$. 
	We will use that $|V^0|\leq |L_R|-1$, cf.\ \cite[Lemma 1]{Hellmuth:15a}.
	Note $\Delta \leq |L_R|$. 
	Moreover, the notation for edges $(u,v)$ is chosen such that  
	$u$ is closer to the root than $v$. 
	\begin{align*}
		B(T) &= \sum_{(u,v)\in E^0} (c(u)-1)(c(v)-1)  
					\leq \sum_{(u,v)\in E^0} \deg(u)(c(v)-1)
					\leq \Delta \cdot \sum_{v\in V^0} (c(v)-1)\\
				&	=\Delta \cdot \Big(- |V^0| + \sum_{v\in V^0} c(v)\Big)
					=\Delta \cdot ( - |V^0| + |E|) \\
				&	=\Delta \cdot (- |V^0|+|V^0|+|L_R|-1 )
					=	\Delta \cdot (|L_R|-1) \in \mc{O}(|L_R|^2).
	\end{align*}	
\end{proof}

\begin{figure}[tbp]
  \begin{center}
    \includegraphics[width=.2\textwidth]{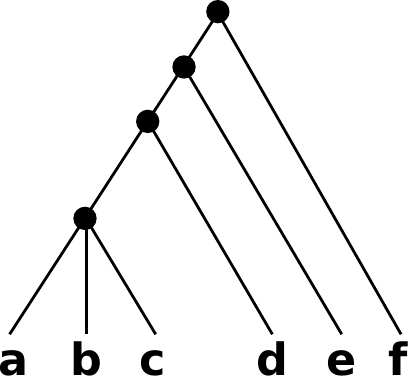}
  \end{center}
	\caption{
		Shown is a tree $T$ on the leaf set $\{a,b,c,d,e,f\}$. 
	 	Let $R_2 = \mathcal{R}(T)$ and  
		$R_1 = \{\rt{ab|d},\rt{ab|e},\rt{ab|f},\rt{bc|e},\rt{bc|f}\}$. 
		Although $R_1\subseteq R_2$, we can observe that for the minimal
	   representative sets $R'_1=R_1 \in\minimal(\SC(R_1))$ and $R'_2 = \{\rt{ab|d},\rt{bc|d},\rt{cd|e},\rt{de|f}\}\in \minimal(\SC(R_2))$
		it holds that $|R'_1|> |R'_2|$.
    }
	\label{fig:exmpl2}
\end{figure}

\begin{cor}
Let $R$ be a triple set that identifies the tree $T$. 
Then, for any $R'\in \minimal(\SC(R))$ we have $|R'| = B(T)\in \mc{O}(|L_R|^2)$. 
\label{cor:idn}
\end{cor}
\begin{proof}
Note, $\mathcal{R}(T)$ is closed and therefore, $\cl(\mathcal{R}(T)) = \mathcal{R}(T)$.
Theorem \ref{thm:boundary} implies that $\cl(R) = \mathcal{R}(T)$. 
Now, if $R'\in \minimal(\SC(R))$ and $R''\in \minimal(\SC(\mathcal{R}(T)))$, 
then $\cl(R') = \cl(R'')$. Hence, we can apply 
Theorem \ref{thm:cor:boundary} and \ref{thm:min=MIN2} to conclude that
$|R'|=|R''| = B(T)$.
\end{proof}

Note, for a rooted binary tree $T$ on $L_R$ and thus, $c(u)=2$ for each interior vertex, 
we obtain $B(T) =|L_R|-2$.  Moreover, $B(T) =|L_R|-2$ shows that $B(T)\in \mc{O}(|L_R|)$ is possible.
On the other hand, if $T$ is tree for which the root $\rho$ has $c(\rho)=n$ children 
and each child of $\rho$ is adjacent to exactly two leaves (and hence, $|E^0|=n$ and  $|L_R|=2n$), 
then we have
$B(T) = n(n-1)$ and therefore, indeed $B(T)\in  \Theta(|L_R|^2)$ and thus, $B(T)\not\in  \mc{O}(|L_R|)$ is possible.

We conjecture that for an arbitrary triple set $R$ and $R'\in \min(\SC(R))$ 
it always holds that  $R'$ is bounded above by $\mc{O}(|L_R|^2)$. However, a main difficulty in proving this is the following
fact: $R_1\subseteq R_2$ and $R'_i\in \min(\SC(R_i))$ with $i=1,2$ does not imply that  $|R'_1|\leq|R'_2|$; see Figure \ref{fig:exmpl2}.

Now consider triple sets $R$ that define a rooted tree $T$ with leaf set $L_R$, that is,  
$T$ is the unique tree (up to isomorphism) that displays $R$ and thus, $T$ must be binary and $\spa(R) = \{T\}$.
In \cite{Steel1992} it was shown that Theorem \ref{thm:min=MIN} is always satisfied for 
triple sets that define a tree. 
\begin{theorem}[{\cite[Cor.\ on Page 111]{Steel1992}}] 
If $R$ is a triple set that defines the rooted tree $T=(V,E)$ with leaf set $L$,
then $|R'|=|L|-2$ for any  $R'\in \minimal(\SC(R))$.
\label{thm:steel}
\end{theorem}
Note, every triple set $R$ that defines a tree $T$ also identifies $T$ and, 
by the discussion above and Corollary \ref{cor:idn}, we can conclude that every minimal triple
set $R$ that defines $T$ must have cardinality $|R| = B(T) =|L_R|-2$.
Thus, Theorem \ref{thm:steel} is an immediate consequence of 
Theorem \ref{thm:min=MIN} and Corollary \ref{cor:idn}.

A further interesting result is given by Semple \cite{SEMPLE03} and Gr{\"u}newald et at.\ \cite{GSS:07}:
\begin{lemma}
	For any subset $R$ of $\mathcal{R}(T)$, $\cl(R) = \mathcal{R}(T)$ if and only if $R$ identifies $T$.

	Furthermore, let $A_R$ denote the unique tree obtained from \texttt{BUILD} with input $R$. 
	For two sets $R_1$ and $R_2$ with $\cl(R_1)=\cl(R_2)$ we have $A_{R_1} = A_{R_2}$. 
	Moreover, if $R$ identifies $T$, then $A_R=T$. 
\label{lem:gruen}
\end{lemma}

The latter result left us 
with the question if the set $\Lmax{R}{R}$
can be used to obtain similar results. In particular, 
the question arises under which conditions  
$\Lmax{R}{R}$ provides the essential information of a hierarchy $\mc{C}(T)$ of $T$. 
Let 
\[\mathbb{C}(R) \coloneqq \bigcup_{\{A,B\} \in \Lmax{R}{R}} \{A,B\}
													\cup L_R \cup \Big\{\{x\} \colon x\in L_R\Big\}.\]

There are many examples that show $\mathbb{C}(R) = \mc{C}(T)$ for some tree $T$, in which case
$\mathbb{C}(R)$ provides already all information to re-build $T$.
As a simple example consider the set  $R_1=\{\rt{ab|c}\}$ where 
$\mathcal{C}(T) = \mathbb{C}(R_1) = \{\{a,b\},\{c\}\} \cup L_{R_1} \cup \{\{a\},\{b\},\{c\}\}$ 
provides the hierarchy of the unique tree $A_{R_1}=\rt{ab|c}$ that displays $R_1$. 
A further example is given in Figure \ref{fig:exmpl123}.
Contrary, for $R_2=\{\rt{ab|d}, \rt{bc|e}\}$ the tree obtained with \texttt{BUILD} is 
$A_{R_2} = ((a,b,c),d,e)$ (given in Newick format). However,  the set  $\mathbb{C}(R_2)$
does not contain the element  $\{a,b,c\}$. Moreover, $\mathbb{C}(R_2)$
contains the elements $\{a,b\}$ and $\{b,c\}$. Hence, $\mathbb{C}(R_2)$ is not a hierarchy.
The difference between $R_1$ and $R_2$ is simple: $R_1$ identifies $A_{R_1}$ and
$R_2$ does not identify $A_{R_2}$. The latter observation leads us to the following result.

\begin{figure}[tbp]
  \begin{center}
    \includegraphics[width=.4\textwidth]{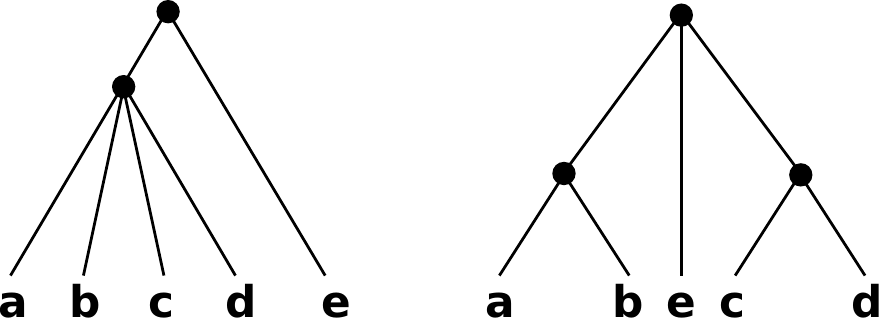}
  \end{center}
	\caption{
		Shown are two trees $T_1$ (left) and $T_2$ (right) that display
		$R = \{\rt{ab|e}, \rt{cd|e}\}$. None of the trees is a refinement of
			the other one. Hence, $R$ neither identifies $T_1$ nor $T_2$. 
		Nevertheless, $\mathbb{C}(R) = \mc{C}(T_2)$ and therefore, contains
		all information to uniquely recover $T_2$. This example also shows that 
 the converse of Statement (2) in Lemma \ref{lem:ident-cr} is not satisfied.
    }
	\label{fig:exmpl123}
\end{figure}

\begin{lemma}
Let $R$ be a consistent triple set and $T\in \spa(R)$. Two vertices $u,v\in V(T)$
form a \emph{pair of siblings $\{u,v\}$}, if $u$ and $v$ have a common adjacent vertex $w$
such that $\mc{C}(u),\mc{C}(v) \subsetneq \mc{C}(w)$ and $|\mc{C}(u)\cup\mc{C}(v)|>2$, i.e., 
$w$ is closer to the root than $u$ and $v$ and at least one of $u$ and $v$ is an inner vertex.
We denote with $\mc S(T)$ the set of all such pairs of siblings in $T$. 
The following statements are satisfied:
\begin{enumerate}
\item  $R$ identifies $T$ if and only if $\Lmax{R}{R} = \bigcup_{\{u,v\} \in \mc S(T)} \{\{\mc C(u),\mc C(v)\}\} $ 
\item  If $R$ identifies $T$, then $\mathbb{C}(R) = \mc{C}(T)$.
\end{enumerate}
\label{lem:ident-cr}
\end{lemma}
\begin{proof}
Let $R$ be a consistent triple set and $T\in \spa(R)$.
By Lemma \ref{lem:gruen}, $R$ identifies $T$ if and only if $\cl(R) = \mc{R}(T)$.

We prove first Item (1). Assume that $R$ identifies $T$.
Let $\lmax{\rt{ab|c}}{R} = \{A,B\} \in  \Lmax{R}{R}$. 
W.l.o.g.\ assume that $a,b\in A$ and $c\in B$. 
By construction of
$\Lmax{R}{R}$  and Theorem \ref{thm:2cc}, we have $\rt{ab|c}\in \cl(R)$.
Note, $\rt{ab|c}\in \cl(R) = \mc{R}(T)$ if and only if there is a pair of siblings $\{u,v\}\in \mc{S}(T)$
such that $a,b\in \mc{C}(u)$ and $c\in \mc{C}(v)$. 

In what follows, we show that $A = \mc{C}(u)$ and $B=\mc{C}(v)$.
Assume for contradiction that $A \not\subseteq \mc{C}(u)$.
Hence, there is an element $x\in A\setminus \mc{C}(u)$. 
By definition of $\lmax{\rt{ab|c}}{R} = \{A,B\}$ and Theorem \ref{thm:2cc}, 
$\rt{ax|c}\in  \cl(R) = \mc{R}(T)$. But this immediately implies that $x\in \mc{C}(u)$;
a contradiction. Hence, $A \subseteq \mc{C}(u)$ and, analogously, $B \subseteq \mc{C}(v)$.
Assume for contradiction that $A \subsetneq \mc{C}(u)$. 
Again, there is an element $x\in\mc{C}(u)\setminus A$, which implies that $\rt{ax|c}\in \mc{R}(T) =\cl(R)$. 
Thus, Lemma \ref{lem:max1} implies that 
there exists a unique element $\lmax{\rt{ax|c}}{R} = \{A',B'\} \in  \Lmax{R}{R}$
such that w.l.o.g.\ $a,x\in A'$ and $c\in B'$.
Thus, $A\neq A'$ as otherwise, $x\in A$. Note, $c\in B\cap B'$. 
However,  Corollary \ref{cor:Bempty} implies that $B\cap B'$ must be empty, 
since $A\cap A'\neq \emptyset$; a contradiction. 
Therefore, $A = \mc{C}(u)$ and, analogously, $B=\mc{C}(v)$.
In summary, $\Lmax{R}{R} \subseteq \bigcup_{\{u,v\} \in \mc S(T)} \{\{\mc C(u),\mc C(v)\}\} $ 

Now, let $\{u,v\} \in \mc S(T)$. Since, $|\mc{C}(u)\cup\mc{C}(v)|>2$ we can assume
that  at least one of $\mc{C}(u)$ or $\mc{C}(v)$ contains at least two elements, say $\mc{C}(u)$. 
Thus, there are $a,b\in \mc{C}(u)$ and $c\in \mc{C}(v)$, which implies that 
$\rt{ab|c}\in \mc{R}(T) = \cl(R)$. Lemma \ref{lem:max1} implies that 
there exists a unique element $\lmax{\rt{ab|c}}{R} = \{A,B\} \in  \Lmax{R}{R}$. 
Now we can re-use exactly the same arguments as before to show that $\mc{C}(u)=A$ and $\mc{C}(v)=B$. 
Thus, $\bigcup_{\{u,v\} \in \mc S(T)} \{\{\mc C(u),\mc C(v)\}\} \subseteq \Lmax{R}{R}$
and therefore,  $\bigcup_{\{u,v\} \in \mc S(T)} \{\{\mc C(u),\mc C(v)\}\}   =\Lmax{R}{R}$

Conversely, assume that  $\Lmax{R}{R} = \bigcup_{\{u,v\} \in \mc S(T)} \{\{\mc C(u),\mc C(v)\}\} $. 
Let $\rt{ab|c}\in \mc{R}(T)$. Again, there must be a pair of siblings $\{u,v\}\in \mc{S}(T)$
such that $a,b\in \mc{C}(u)$ and $c\in \mc{C}(v)$. Hence, $\{\mc{C}(u),\mc{C}(v)\} \in \Lmax{R}{R}$. 
Since $\lmax{\rt{ab|c}}{R} \in \LrR{\rt{ab|c}}{R}$, we can apply Lemma \ref{lem:nonempty} to
conclude that $\rt{ab|c}\in \cl(R)$ and hence, $\mc{R}(T)\subseteq \cl(R)$. 
Moreover,  $T\in \spa(R)$ implies that  $\cl(R)\subseteq \mc{R}(T)$ and therefore, 
$\mc{R}(T) =  \cl(R)$. Lemma \ref{lem:gruen} implies that $R$ identifies $T$.

We continue to prove Item (2). 
Assume that $R$ identifies $T$. 
Hence,  $\Lmax{R}{R} = \bigcup_{\{u,v\} \in \mc S(T)} \{\{\mc C(u),\mc C(v)\}\} $. 
Now it is easy to see that $\mc{C}^*\coloneqq\bigcup_{v\in V^0\setminus \{\rho_T\}} \{\mc C(v)\} =\bigcup_{\{u,v\} \in \mc S(T)} \{\mc C(u),\mc C(v)\}$.
Since $L_T= L_R$ we obtain 
$\mc{C}(T) = \mc{C}^* \cup L_T \cup \left\{\{x\} \colon x\in L_T\right\} = \mathbb{C}(R)$.
\end{proof}

Note, the converse of Statement (2) in Lemma \ref{lem:ident-cr} is not satisfied in 
general, see Figure \ref{fig:exmpl123}.

\subsection{Quartets}

Here, we consider unrooted trees in which every inner vertex has degree at least $3$.
Splits and quartets (unrooted binary trees on four leaves) serve as building blocks for unrooted trees.
To be more precise, each edge $e\in E(T)$ of an unrooted tree $T$ gives rise to a split $A|B$, that is, if one
removes $e$ from $T$ one obtains two distinct trees $T_1$ and $T_2$ with
leaf sets $A= L(T_1)$ and $B=L(T_2)$. 
A tree can be reconstructed in linear time from its set of splits \cite{ME:85,Gus:91,BUNEMAN1974205}

If there is a split $A|B$ in $T$ such that 
$a,a'\in A$ and $b,b'\in B$, we say that the quartet $aa'|bb'$
is displayed in $T$. Equivalently, the quartet $aa'|bb'$ is displayed in 
 $T$,
if $a,a',b,b'\in L(T)$ and the path from $a$ to $a'$ does not intersect the path from $b$ to $b'$ in $T$. 
If $Q$ contains all  quartets that are displayed in an unrooted tree $T$, then 
$T$ is uniquely determined by $Q$ and can be reconstructed in polynomial time \cite{RSA:RSA3}. 
An arbitrary set of quartets $Q$ is called consistent,
if there is a tree that displays each quartet in $Q$, see \cite{Semple:book,Steel:book} for further details. 
Determining whether an arbitrary set of quartets is consistent is an NP-complete
problem \cite{Steel1992}.

Analogously as for rooted triples, we can define
the set $\spa(Q)$, the closure $\cl(Q)$ and the two sets $\minimal(\SC(Q))$ and $\minimum(\SC(Q))$.  
Now, consider the ordered pair $(Q,\mathbb{F}_Q)$ where $Q$ is a consistent set of quartets
and $\mathbb{F}_Q =\{Q''\subseteq Q'\colon Q'\in \minimal(\SC(Q))\}$.
Of course, one might ask whether $(Q,\mathbb{F}_Q)$ is a matroid as well and thus, 
whether minimal representative sets $Q'\in \minimal(\SC(Q))$ have all the same 
cardinality.

A counterexample, which we recall here for the sake of completeness,  
is given in \cite{Steel1992}:
Let $\mc Q(T)$ be the set of all quartets that are displayed in 
a binary unrooted tree $T$ with leaf set $L$. Thus,  $\spa(\mc Q(T)) = \{T\}$.  
Proposition 2(3) in \cite{Steel1992} implies that there is a minimal subset 
$Q'\subseteq \mc Q(T)$ of size $|Q'|=|L|-3$ such that $\spa(Q') = \{T\}$
and hence, $\cl(\mc Q(T))=\cl(Q')$. 
Thus, for the tree $T$ in Figure \ref{fig:quartet} there is a minimal 
representative quartet set of size $4$. However, the set 
$Q' = \{{57|24, 15|67, 12|35, 47|13, 34|56}\}$ is also a minimal 
representative quartet set of $Q(T)$, but has size $5$. Thus, the basis elements of 
 $(Q,\mathbb{F}_Q)$ don't have the same size, in general. Therefore, 
 $(Q,\mathbb{F}_Q)$ is not a matroid.

\begin{figure}[htbp]
  \begin{center}
    \includegraphics[width=.3\textwidth]{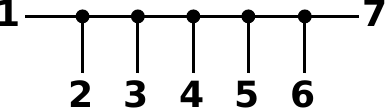}
  \end{center}
	\caption{
		Shown is a binary unrooted tree $T$ with leaf set $L=\{1,2,\dots,7\}$.
    }
	\label{fig:quartet}
\end{figure}

\section{Conclusion and Outlook}
\label{sec:end}

In this contribution, we were concerned with minimum representative triple
sets, that is, subsets $R'\subseteq R$ that have minimum cardinality and
for which $\cl(R') = \cl(R)$. We have shown that it is possible to compute
minimum representative triple sets in polynomial time via a simple greedy
approach. To prove the correctness of this method, we showed that minimal
representative sets (and its subsets) form a matroid $(R,\mathbb{F}_R)$.
Minimal representative sets contain minimum representative sets and since
they form the basis of the matroid $(R,\mathbb{F}_R)$, they all must have
the same cardinality. The techniques we used to show the matroid structure
have been utilized to provide a novel and efficient method to compute the
closure $\cl(R)$ of a consistent triple set $R$. For this algorithm,
minimum representative triple sets $R'\in \minimal(\SC(R))$ can be used as
input, which significantly improves the runtime of the closure computation.
Hence, a particular problem that might be addressed in future work is the
design of a more efficient algorithm to compute $R'\in \min(\SC(R))$.
Furthermore, the size of $R'\in \minimal(\SC(R))$ is not known \emph{a
priori}. Boundaries for such sets $R'$ have not been established so-far,
except for some rare examples as ``defining'' or ``identifying'' triple
sets \cite{GSS:07,Steel1992}. Thus, in order to understand minimal
representative triple sets in more detail, a more thorough analysis of the
structure of the matroid $(R,\mathbb{F}_R)$, its collection of bases $\min(\SC(R))$ or
its dual $(R,\mathbb{F}_R)^*$ is needed. 

We also assume that the runtime of Algorithm \ref{alg:lmax} can be
improved, which would immediately lead to a faster method to compute
$\cl(R)$.

An interesting starting point for future research might be the investigation 
of the sets $\Lmax{R}{R}$ in more detail and finding a characterization 
for sets $R$ where $\mathbb{C}(R)$ provides a hierarchy $\mathcal{C}(T)$ of some tree $T$.

Moreover, generalizations of the established results would be of interest, for instance, 
is there still a matroid structure if one does not insist that for the subset $R'$ of $R$ we have 
$\cl(R')=\cl(R)$?
What can be said about the structure of representative sets for non-consistent
triple sets, see e.g.\ \cite{GSS:07}?
Although minimal representative sets of quartets do not provide a matroid structure, 
it might be useful to figure out which of the other established result are satisfied
for quartets as well.  

\section*{Acknowledgment}
We are grateful to Volkmar Liebscher, Mike Steel,   Annemarie Luise K\"uhn
and the anonymous referees for their constructive comments and 
suggestions which has led to a significant improvement of this paper.

\bibliographystyle{plain}
\bibliography{biblio}
\end{document}